\documentclass[10.5pt]{article}
\usepackage[margin=1.3in]{geometry}
\usepackage{amsfonts}
\usepackage{graphicx}
\usepackage{epstopdf}
\usepackage{pgfplots}
\usepackage{tikz}
\usepackage{algorithm}
\usepackage{enumerate}
\usepackage{comment}
\usepackage{algpseudocode}
\usepackage{amsmath}
\usepackage{pgfplotstable}
\usepackage{booktabs}
\usepackage{amsthm}
\newtheorem{theorem}{Theorem}[section]

\newtheorem{proposition}[theorem]{Proposition}
\newtheorem{definition}[theorem]{Definition}

\newtheorem{example}[theorem]{Example}
\newcommand{\N}{\mathcal{N}}
\newcommand{\NN}{\mathbb{N}}

\newcommand{\id}{\textup{id}}
\newcommand{\nulll}{\textup{null}}
\newcommand{\im}{\textup{im}}
\newcommand{\m}{\mathbf{m}}

\newcommand{\sub}{\mathcal{S}}

\newcommand{\Res}{\textup{Res}}
\newcommand{\reg}{\textup{reg}}
\newcommand{\rank}{\textup{rank}}

\newcommand{\diag}{\textup{diag}}

\newcommand{\I}{I}
\newcommand{\Rh}{S}
\usepackage{url}
\newcommand{\OO}{\mathcal{B}}

\newcommand{\Z}{\mathbb{Z}}

\newcommand{\C}{\mathbb{C}}
\newcommand{\PP}{\mathbb{P}}

\newcommand{\V}{\mathbb{V}}
\newcommand{\1}{\mathbf{1}}
\newcommand{\x}{x}
\newcommand{\D}{\delta}

\newcommand{\R}{\mathbb{R}}
\usepackage{tikz}
\usepackage{pgfplots}
\newcommand{\Vol}{\textup{Vol}}
\newcommand{\Span}{\textup{span}}
\newcommand{\MV}{\textup{MV}}
\newcommand{\ideal}[1]{\langle{#1}\rangle}

\usepackage{kbordermatrix}
\usepackage{tikz-cd}
\pgfplotsset{
  log x ticks with fixed point/.style={
      xticklabel={
        \pgfkeys{/pgf/fpu=true}
        \pgfmathparse{exp(\tick)}%
        \pgfmathprintnumber[fixed relative, precision=3]{\pgfmathresult}
        \pgfkeys{/pgf/fpu=false}
      }
  }}
\begin{document}
\title{Solving Polynomial Systems via a Stabilized Representation of Quotient Algebras}
\author{Simon Telen, Bernard Mourrain, Marc Van Barel\thanks{Supported by
  the Research Council KU Leuven,
PF/10/002 (Optimization in Engineering Center (OPTEC)),
C1-project (Numerical Linear Algebra and Polynomial Computations),
by
the Fund for Scientific Research--Flanders (Belgium),
G.0828.14N (Multivariate polynomial and rational interpolation and approximation),
   and by
   the Interuniversity Attraction Poles Programme, initiated by the Belgian State,  Science Policy Office,
   Belgian Network DYSCO (Dynamical Systems, Control, and Optimization).
}}
\maketitle

\begin{abstract}
We consider the problem of finding the isolated common roots of a set of polynomial functions defining a zero-dimensional ideal $I$ in a ring $R$ of polynomials over $\C$. We propose a general algebraic framework to find the solutions and to compute the structure of the quotient ring $R/I$ from the null space of a Macaulay-type matrix. The affine dense, affine sparse, homogeneous and multi-homogeneous cases are treated. In the presented framework, the concept of a border basis is generalized by relaxing the conditions on the set of basis elements. This allows for algorithms to adapt the choice of basis in order to enhance the numerical stability. We present such an algorithm and show numerical results.
\end{abstract}

\section{Introduction}
There exist several methods to find all the roots of a set of polynomial equations \cite{sturmfels2,cattani2005solving}. The most important classes are homotopy continuation methods \cite{bates2013numerically,verschelde1999algorithm}, subdivision methods \cite{mourrain2009subdivision} and algebraic methods \cite{emir1,mvb,cox2,dreesen2012back,mourrain1999new,telen2017stabilized}. In this paper, we focus on the latter class of solvers.

These methods perform linear algebra operations on vector subspaces of
the ideal generated by the set of equations to deduce the algebraic
structure of the quotient algebra of the polynomial ring by the ideal.
One can find the roots of these techniques in ancient works on
resultants by B\'ezout, Sylvester, Cayley, Macaulay\dots. Explicit
constructions of matrices of polynomial coefficients are exploited to
compute projective resultants of polynomial systems
(see e.g., \cite{macaulay_formulae_1902}). These matrix constructions have been
further investigated to compute other types of resultants such as toric
or sparse resultants \cite{cox2, emir1, dandrea_poisson_2015} or
residual resultants \cite{buse_resultant_2001}. See e.g. \cite{emiris_matrices_1999} for
an overview of these techniques. These matrices are also exploited in
numerical linear algebra-based methods for finding the solutions of
the polynomial equations from their null space \cite{dreesen2012back,
  telen2017stabilized}.

Another well-established approach to describe the quotient algebra structure is by
computing Groebner bases for a given monomial ordering \cite{cox1}. The initial
algorithms based on rewriting techniques have been enhanced by
introducing linear algebra tools \cite{faugere1999new,
  eder_signature-based_2011}. H-bases, initiated by F.S. Macaulay,
have also been investigated to construct ideal bases, with interesting
projection properties to compute normal forms and to describe quotient
algebras \cite{moller_h-bases_2000}. To avoid the numerical
instability induced by monomial orderings in Groebner bases
computations, border bases have been developed to combine robustness
and efficiency \cite{mourrain1999new, mourrain_generalized_2005,
 mourrain_stable_2008}.

These methods proceed incrementally by performing linear algebra
operations on monomial multiples of polynomials computed at the
previous step, until a reduction or a commutation property is
satisfied. The sizes of the matrices involved in these computations
are usually smaller than the size of resultant matrices (see
e.g. \cite{mourrain_solving_2000}). Because of the incremental nature
of these methods, the computed bases describing the quotient algebra
structure may not be optimal, from a numerical point of view.

The framework we consider in this paper is related to the construction
of ideal interpolation (or normal form).  In \cite{de_boor_ideal_2004,
  de_boor_ideal_2006}, the problem of characterizing when a linear
projector is an ideal projector, that is when the kernel of the
projector is an ideal, is investigated.  The conditions of commutation
and the connectivity property of the basis proposed in 
\cite{mourrain1999new} are discussed and compared to some variants.

In this paper, we propose a new method to compute the solutions and
the algebra structure of the coordinate ring from the null space of
Macaulay-type matrices. Such a null space is the orthogonal space of a
vector subspace of the ideal $I$ generated by the set of polynomial
equations. We give new conditions on this null space (Theorem
\ref{thm:affine}) under which the quotient algebra structure can be
recovered and propose a new method to compute it and to solve the
polynomial equations. Implicitly, this generalizes the concept of
border bases \cite{mourrain1999new, mourrain_generalized_2005} by
relaxing the conditions of connectivity on the basis. It also
characterizes when the null space defines a projector, which is the
restriction of an ideal projector, as studied in
\cite{de_boor_ideal_2004, de_boor_ideal_2006}.

We show how to construct such matrices and to find the roots for generic dense,
sparse, homogeneous and multi-homogeneous systems. For homogeneous
systems, these conditions lead to a new criterion of regularity of the
ideal $I$ (Proposition \ref{prop:regularity}), extending in the
zero-dimensional case, the criterion proposed in
\cite{bayer1987criterion}.

In \cite{telen2017stabilized} it is shown that the choice of basis of
the coordinate ring is crucial to the numerical stability of algebraic
solving methods.  In the framework we propose here, an algorithmic
`good' choice of basis can be made.

The methods we propose in this paper are numerical linear algebra
methods using finite precision arithmetic. Groebner bases methods
require symbolic computation because they are unstable. This makes
these methods unfeasible for large systems. We compare our algorithms
to homotopy continuation methods in double precision, because these
methods are known to be successful numerical solvers
\cite{bates2013numerically,verschelde1999algorithm}.
However, we show in our numerical experiments that these methods
do not guarantee that all solutions are found. On the contrary, our
methods do find all solutions under some genericity assumptions, and
they are competitive in speed when the number of variables is not too
large.

Throughout the paper, we assume zero-dimensionality of the ideal
generated by the input equations.
We start with a motivation in the next section. In Section
\ref{sec:affine} we treat the rootfinding problem in affine space. We
assume that the number of solutions in $\C^n$ is finite. Theorem
\ref{thm:affine} is the main theorem of the paper, since the results
in other sections follow from it. In the case of a dense set of
equations, the approach in \cite{dreesen2012back} follows from Theorem
\ref{thm:affine}. Section \ref{sec:toric} deals with the toric case:
we assume a finite number of solutions in the algebraic torus
$(\C^*)^n$. In Section \ref{sec:proj}, we consider the case of dense
systems in a projective setting. We use the framework to compute a
representation of the degree $\rho$ part of the quotient algebra,
where $\rho$ is the regularity of the ideal. We assume a finite number
of solutions in $\PP^n$. Section \ref{sec:multihom} treats the
multihomogeneous case, where we have $\D$ solutions in
$\PP^{n_1} \times \cdots \times \PP^{n_k}$. In every setting, we
present an appropriate Macaulay-type matrix to work with in the case
of a square system (as many equations as the dimension of the solution
space). In Section \ref{sec:tablestosolutions} we elaborate on how to
find the solutions from a representation of the quotient
algebra. Finally, in Section \ref{sec:numexp} we show some numerical
examples. We assume that the number of solutions, counting
multiplicities, is $\D$ and we denote them by
$z_i, i = 1,\ldots,\D \in \star$ where $\star$ is the solution space.

\section{Normal forms in an Artinian ring}
Let $R = \C[x_1,\ldots, x_n]$ be the ring of polynomials in the
variables $x_{1},\ldots,x_{n}$ with coefficients in the field $\C$ and
take $I \subset R$ defining $\D < \infty$ points, counting
multiplicities, such that $R/I$ is Artinian. Equivalently, we assume
that $\dim_{\C}(R/I)=\delta$. 
\begin{definition}
A \textup{normal form} on $R$ w.r.t. $I$ is a linear map $\N: R \rightarrow B$ where $B \subset R$ is a vector subspace of dimension $\D$ over $\C$ such that 
\begin{center}
\begin{tikzcd}[column sep = scriptsize, row sep = 0.2 cm]
    0 \arrow{r} & I \arrow{r} & R \arrow{r}{\N} & B \arrow{r} & 0\\
\end{tikzcd}
\end{center}
is exact and $\N_{|B} = \id_B$.
\end{definition}
In \cite{de_boor_ideal_2004}, $\N$ is also called an ideal projector.

Let $\N$ be a normal form on $R$ w.r.t. $I$. We restrict $\N$ to a subspace $V \subset R$ such that $B \subset V$ and $x_i \cdot B \subset V, i = 1,\ldots,n$. Let $P : B \rightarrow \C^\D$ be an isomorphism defining coordinates on $B$. Defining $N = P \circ \N$ and $N_i : B \rightarrow \C^\D$ given by $N_i(b) = N(x_i \cdot b)$, we have the following facts: 
\begin{enumerate}
\item $\ker(N) = I \cap V$,
\item $N_{|B} = P$,
\item $m_{x_i}(b) = \N(x_i \cdot b) = (P^{-1} \circ N_i)(b)$.
\end{enumerate}
Notice that an $N$ satisfying these properties, is of the form $N: f\in V \rightarrow N(f) = (\eta_1(f), \ldots, \eta_\D(f)) \in \C^{\D}$ with $\eta_i \in V^* \cap I^{\perp} =\{ \lambda\in V^*\mid \forall p \in I\cap V, \lambda(p)=0\}$. In other words, $N$ is given  by $\D$ linear forms, which vanish on $I\cap V$.

In this paper we focus on the problem of characterizing when a map $N$
given by $\D$ linear forms
$\eta_{1},\ldots, \eta_{\D}\in V^* \cap I^{\perp}$ is the restriction
of a normal form w.r.t $I$. Given an ideal $I \subset R$ defining an
Artinian algebra $R/I$ of dimension $\D$ and a linear map
$N: V \rightarrow \C^\D$ such that $\ker N\subset I$, we want to
determine when there exists a normal form $\N$ w.r.t. $I$, which by
restriction gives $N$.  More precisely, we determine necessary and
sufficient conditions on $N$ and $V$ such that $N = P \circ \N$, where
$\N$ is the restriction to $V$ of a normal form w.r.t. $I$,
and where $P = N_{|B}$ is invertible with $B \subset V$ such that
$x_i \cdot B \subset V, i = 1,\ldots, n$.

It is clear that if we have this property and we can compute $P$, then the algebra structure of $R/I$ is defined by the multiplication tables $P^{-1} \circ N_i$ and we can use this to find the roots of $I$.

In the following sections, we show how the result can be applied in the affine, toric, homogeneous and multihomogenous setting and propose a numerical construction of $N$ in the case where $I$ is a complete intersection. 

\section{Ideals defining points in $\C^n$} \label{sec:affine}
Denote $R = \C[x_1,\ldots, x_n] = \C[\x]$. We consider a 0-dimensional
ideal $I = \ideal{ f_1, \ldots, f_s } \subset R$ generated by
$s$ polynomials $f_{1},\ldots, f_{s}$ in $n$ variables with $\D < \infty$ solutions in $\C^n$, counting multiplicities. Let $V$ be the vector space of polynomials in $R$ supported in some finite subset $\sub$ of $\NN^n$
$$ V = \bigoplus_{\alpha \in \sub} \C \cdot x^\alpha \subset R,$$
such that $V$ contains at least one unit $u$ in $R/I$ (for instance $u$ could be $1$).
Suppose we have a linear map 
$$N :  V \longrightarrow \C^\D,$$
such that $\ker(N) \subset I \cap V$. For an ideal $J\subset R$ and
$p\in R$, we denote $(J:p)=\{q \in R\mid pq \in J \}$ and $(J:p^{*})=\{q \in R\mid
\exists k\in \NN\ s.t.\ p^{k}q \in J \}$.

\begin{theorem} \label{thm:affine}
Assume that $\dim_{\C}R/I=\D$, that $\ker (N)\subset I \cap V$ and that $V$ contains an element $u$ invertible in $R/I$.
If there is a vector subspace $W \subset V$ such that $x_i \cdot W \subset V, i = 1, \ldots, n$ and for the restriction of $N$ to $W$ we have $\rank(N_{|W}) = \D$, then for any vector subspace $B \subset W$ such that $W = B \oplus \ker(N_{|W})$, we have: 
\begin{enumerate}[(i)]
\item $N^* = N_{|B}$ is invertible, \label{it1}
\item there is an isomorphism of $R$-modules $B \simeq R/I$,  
\item $V= B \oplus V\cap I$ and $I=(\ideal{ \ker (N) }:u)$,
\item the maps $N_i$ given by \begin{alignat*}{2}
	 N_i : ~ & B \longrightarrow \C^\D,  \\
	 &b \longrightarrow N(x_i\cdot b) 
\end{alignat*}
for $i = 1,\ldots,n$ can be decomposed as $N_i = N^* \circ m_{x_i}$
where $m_{x_i}: B \rightarrow B$ define the multiplications by $x_i$
in $B$ modulo $I$ and are commuting ($m_{x_{i}}\circ\, m_{x_{j}}=
m_{x_{j}}\circ\, m_{x_{i}}$ for $1\le i<j\le n$). 
\end{enumerate}
\end{theorem}
\begin{proof}
\begin{enumerate}[(i)]
\item Since $N_{|W}: W \rightarrow \C^\D$ is surjective and $W = B \oplus \ker(N_{|W})$, we have $N_{|B} : B \rightarrow \C^\D$ invertible. 
\item It follows from (i) that $V = B \oplus \ker(N)$. Let $\pi: V \rightarrow B$ be the projection onto $B$ along $\ker(N)$ and define 
\begin{alignat*}{2}
	 m_{x_i} : ~ & B \longrightarrow B,  \\
	 &b \longrightarrow \pi(x_i \cdot b) .
\end{alignat*}
Then $\forall b \in B$,
\begin{eqnarray}
m_{x_i}(b) &=& x_i \cdot b \mod \ker(N)\label{eq:modker}\\
&=& x_i \cdot b \mod I \label{modeq}
\end{eqnarray} 
where the last equality follows from $\ker(N) \subset I \cap V$.

For $\alpha \in \NN^n$, we write $\m^\alpha = m_{x_1}^{\alpha_1} \circ \cdots \circ m_{x_n}^{\alpha_n}$ and for $f = \sum_{i=1}^p c_i x^{\alpha_i} \in R$ we define $$f(\m) = \sum_{i=1}^p c_i \m^{\alpha_i} : B \rightarrow B.$$ 
Replacing $u$ by $\pi(u)$ which is also invertible in $R/I$, we can assume that $u\in B$.

We show now that the sequence
\begin{center}
\begin{tikzcd}[column sep = scriptsize, row sep = 0.2 cm]
    0 \arrow{r} & J \arrow{r} & R \arrow{r}{\phi} & B \arrow{r} & 0\\
                    & & f \arrow{r}{}& f(\m)(u) &
\end{tikzcd}
\end{center}
with $ J = \ker(\phi)$ is exact, that is, $\phi(R)=B$ and that $J=I$.
The relation \eqref{modeq} implies that $\forall f\in R, \phi(f)\equiv f \,u  \mod I$ so that $J=\ker \phi \subset I$.
If $\pi_I: R \rightarrow R/I$ is the map that sends $f$ to its residue
class in $R/I$, we have $\pi_I(\phi(f)) = \pi_I(f\, u)$. Hence
$\pi_I(\phi(R)) =\pi_{I}(R\,u)= R/I$ since $u$ is invertible in $R/I$
and $\dim_\C(\phi(R)) \geq \dim_\C(R/I) = \D$. But also $\phi(R)
\subset B$ means $\dim_\C(\phi(R)) \leq \dim_\C(B) = \D$. We deduce
that $\phi$ is surjective and $\pi_I:B \rightarrow R/I$ is an
isomorphism. It follows that the induced map $\overline{\phi}:
R/J\rightarrow  B \simeq R/I$ is an isomorphism of $\C$-vector spaces, which implies $J=I$ since $J\subset I$. We conclude that $\overline{\phi}$ is an isomorphism of $R$-modules between $R/I$ and $B$ and its inverse is $u^{-1} \cdot \pi_{I}$.
This proves the second point. 

\item Moreover, $B \cap I = \{0\}$ since $\pi_{I}: B \rightarrow R/I$
  is an isomorphism; As $B$ is supplementary to $\ker(N)$ in $V$, we
  deduce that $V\cap I = \ker(N)$. It follows that
  $V= B\oplus \ker(N)=B \oplus V\cap I$. We have $\ker (N) \subset I$
  and thus $\ideal{ \ker (N) } \subset I$. To prove the reverse inclusion, notice
  that if $f\in I=J=\ker \phi$ then by the relation
  \eqref{eq:modker}, $f \, u\in \ideal{ \ker (N) } $. This implies that
  $$
  I \subset (\ideal{ \ker (N) }:u) \subset I:u=I 
  $$
  since $u$ is invertible modulo $I$.
  This proves the third point.

\item From Equation \eqref{modeq} and the isomorphism
  $\overline{\phi}$ between $R/I$ and $B$, we deduce that the operators $m_{x_{i}}$
  correspond to the multiplications by the variables $x_{i}$ in the
  quotient algebra $R/I$. Thus they are commuting. By construction, we have
  $N_i(b) = N(x_i \cdot b) = N(\pi(x_i \cdot b)) = (N^* \circ
  m_{x_i})(b)$, where the second equality follows from $\ker(\pi) =
  \ker(N)$. This concludes the proof of the fourth point.

\end{enumerate} 

\end{proof}
It follows from Theorem \ref{thm:affine} that once we have a matrix
representation of $N^*$ and the $N_i, i = 1, \ldots, n$, the matrices
$m_{x_i}$ are given by $(N^*)^{-1} N_i$. The eigenvalues
$z_{ji}, j = 1, \ldots, \D$ can be computed as the generalized
eigenvalues of $N_i v = \lambda N^* v$. As detailed in Section
\ref{sec:tablestosolutions}, computing the eigenvalues and
eigenvectors of the operators of multiplication yields the solution of the polynomial equations.

When $u=1\in V$, then $\forall b \in B, \phi(b)\equiv b \mod I$. Since $B\cap I=\{0\}$, we have $\forall b\in B$, $\phi(b)=b$ and $\phi$ is the normal form or
ideal projector on $B$ along its kernel $I$. Moreover, (iii) implies that $\ideal{ \ker (N) } =I$.

By the normal form characterization proved in
\cite{mourrain1999new, mourrain_generalized_2005}, if the set $B$ is connected to $1$ ($1\in B$ and there exists vector
spaces $B_{l}\subset R$ such that $B_{0}=\Span(1) =\C \subset B_{1}\subset \cdots \subset B_{k}=B$ with
$B_{l+1}\subset B_{l}^{+}$ where $B_{l}^{+}=B_{l}+ x_{1} B_{l}+
\cdots + x_{n}B_{l}$), then the
commutation property (point (iv)) implies that $B\sim R/I$ (point (ii)).

\subsection{Constructing $N$ for dense square systems}
Consider a zero-dimensional ideal $I = \ideal{ f_1, \ldots, f_n }
\subset R$ such that the $f_i$ define a system of polynomial equations
that has no solutions at infinity. That is, denoting $\deg(f_i) =
d_i$, we assume that the system $\{f_1, \ldots, f_{n}\}$ is generic in the sense that there are $\D = \prod_{i=1}^n d_i$ solutions, counting multiplicities, in $\C^n$. We denote these solutions by $\V(I) = \{z_1,\ldots, z_{\D_0} \} \subset \C^n$, where $\D_0 \leq \D$ is the number of distinct solutions. Next, we consider a generic linear polynomial $f_0$.

We use the classical Macaulay resultant matrix construction defined as follows. Let $\rho = \sum_{i=1}^n d_i - n +1$, let $V=R_{\leq \rho}$ be the space of polynomials of degree $\leq \rho$ and $V_{i}=R_{\leq \rho-d_{i}}$. The associated resultant map is
\begin{eqnarray*}
	 M_0 : \quad  V_{0}\times V_1 \times \cdots \times V_n& \longrightarrow& V  \\
	 (q_{0},q_1, \ldots, q_n) &\longmapsto & q_{0}f_{0}+ q_1f_1+ \cdots + q_nf_n.
\end{eqnarray*}
There is a square submatrix $M'$ of the matrix of $M_0$ such that $\det(M')$ is a nontrivial multiple of the resultant $\Res(f_0,f_1,\ldots,f_n)$ \cite{cox2,macaulay1994algebraic}.
In the notation of \cite{telen2017stabilized}, the monomial multiples of $f_{0}$ involved in $M'$ are with exponents in $\Sigma_0 = \{\alpha \in \NN^n: \alpha_i < d_i, i = 1,\ldots,n \}$. The set $\OO_{0}$ of monomials with exponents  in $\Sigma_0$ corresponds generically to a basis (the so-called Macaulay basis) of $R/I$: $B_0 = \Span(\OO_0) \simeq R/I$. 
The matrix $M'$ decomposes as $$M' = \begin{bmatrix}
M_{00} & M_{01} \\ M_{10} & M_{11}
\end{bmatrix}$$
where the rows and columns of the first block $M_{00}$ are indexed by $\OO_{0}$.
The matrix $\tilde{M}=\begin{bmatrix} M_{01}\\ M_{11} \end{bmatrix}$ representing monomial multiples of $f_{1}, \ldots, f_{n}$ is such that $\im(\tilde{M})\subset I\cap V$.
Since for generic systems $f_{1}, \ldots, f_{n}$, the matrix $M_{11}$
is invertible for a generic system $f=(f_{1},\ldots,f_{n})$ (see \cite{macaulay1994algebraic}, \cite[Chapter 3]{cox2}), the rank of $\tilde{M}$ is $\dim V -\D$.
Let $N$ be the coefficient matrix of a basis of the left null-space of $\tilde{M}$ so that $N\, \tilde{M}=0$.
Then $N$ corresponds to a linear map $V\rightarrow \C^{\delta}$ of rank $\delta$ such that its kernel is $\im(\tilde{M})\subset I$.
In fact, denoting $M = (M_0)_{|V_1 \times \cdots \times V_n}$ (i.e. $M(q_1,\ldots,q_n) = q_1f_1 + \ldots + q_nf_n$) it satisfies  
$$
\ker(N) = \im(\tilde{M}) = \im(M) = I \cap V= I_{\leq \rho},
$$
since $B_0 \cap I=\{0\}$ and $M_{11}$ is invertible, so that any element in $\im(M)$ can be projected in $B_0 \cap I$ along $\im(\tilde{M})$ (i.e. $\im(M)\subset \im(\tilde{M}) \subset \im(M)$).

In order to apply Theorem \ref{thm:affine}, we need to restrict $N$ to a subset $W \subset V$, such that $x_i \cdot W \subset V$ and $N_{|W}$ is surjective. Let us take $W= R_{\leq \rho-1}$. Since $M_{11}$ is invertible, $N$ is equivalent to the matrix
$\begin{bmatrix} \id & -M_{01}M_{11}^{-1} \end{bmatrix}$ where the columns of the $\delta \times \delta$ identity block are indexed by the monomials in $\OO_{0}$. Since $B_{0}\subset W$, we deduce that $N_{|W}$ is surjective.


This leads to Algorithm \ref{alg:affine} for computing the algebra structure of $R/I$. Note that in step \ref{basischoice} of the algorithm we make a choice of monomial basis for $R/I$. In order to have accurate multiplication matrices, $N^*$ should be `as invertible as possible'. A good choice here is to use QR with optimal column pivoting on the matrix $N_{|W}$, such that $\OO$ corresponds to a well-conditioned submatrix. This technique is used for the choice of basis on $M$ in \cite{telen2017stabilized}. We use $M$ instead of $\tilde{M}$ for numerical reasons. It leads to a more accurate computation of the null space. 
\begin{algorithm}[ht!]
\caption{Computes the structure of the algebra $R/I$ (affine, dense case)}\label{alg:affine}
\begin{algorithmic}[1]
\Procedure{AlgebraStructure}{$f_1,\ldots,f_n$}
\State $M \gets \textup{the resultant map on $V_1 \times \cdots \times V_n$}$ \label{constmac}
\State $N \gets \nulll(M^{\top})^\top$ \label{nullspace}
\State $N_{|W} \gets \textup{columns of $N$ corresponding to monomials of degree $< \rho$}$ 
\State $N^* \gets \textup{columns of $N_{|W}$ corresponding to an invertible submatrix}$ \label{basischoice}
\State $\OO \gets \textup{monomials corresponding to the columns of $N^*$}$ 
\For{$i=1,\ldots,n$}
\State $N_i \gets \textup{columns of $N$ corresponding to $x_i \cdot \OO$}$
\State $m_{x_i} \gets (N^*)^{-1}N_i$
\EndFor
\State \textbf{return} $m_{x_1},\ldots, m_{x_n}$
\EndProcedure
\end{algorithmic}
\end{algorithm}

\begin{example} \label{ex:affine}
Consider the ideal $I = \ideal{ f_1,f_2 } \subset \C[x_1,x_2]$ given by 
\begin{eqnarray*}
f_1 &=& 7 + 3x_1 -6 x_2 -4x_1^2 +2 x_1x_2 + 5 x_2^2,\\
f_2 &=& -1 - 3x_1 +14 x_2 -2 x_1^2 +2x_1x_2-3x_2^2.
\end{eqnarray*}
As illustrated in Figure \ref{fig:ex2}, the solutions are $z_1 = (-2,3), z_2 = (3,2), z_3 =(2,1), z_4 =(-1,0)$. The dense Macaulay matrix $M$ of degree $\rho = d_1+d_2 - n+1 = 3$ is 
\[
  M^\top=\kbordermatrix{%
      & 1  & x_1 & x_2 & x_1^2 & x_1x_2 & x_2^2 &x_1^3 & x_1^2x_2 & x_1x_2^2& x_2^3  \\
    f_1 & 7  & 3 & -6& -4 & 2 & 5&&&& \\
    x_1f_1 & &7 & &3&-6& & -4&2&5\\
    x_2f_1 & & & 7 & & 3 & -6& & -4 &2&5\\
    f_2 & -1  & -3 & 14 & -2&2&-3 \\
    x_1f_2 & &-1 & &-3&14& & -2&2&-3\\
    x_2f_2 & & & -1 & & -3 & 14& & -2 &2&-3\\
  }.
\]
Since all solutions are simple, a basis for the left null space of $M$ is given by $v^{(3)}(z_i), i = 1,\ldots,4$, where
$$ v^{(3)}(x_1,x_2) = \begin{bmatrix}
1  & x_1 & x_2 & x_1^2 & x_1x_2 & x_2^2 &x_1^3 & x_1^2x_2 & x_1x_2^2& x_2^3
\end{bmatrix}.$$
We find 
\[
  N=\kbordermatrix{%
     &1  & x_1 & x_2 & x_1^2 & x_1x_2 & x_2^2 &x_1^3 & x_1^2x_2 & x_1x_2^2& x_2^3  \\
     v^{(3)}(-2,3)& 1 & -2 & 3 &4&-6&9&-8&12&-18&27\\
     v^{(3)}(3,2)&1&3&2&9&6&4&27&18&12&8\\
     v^{(3)}(2,1)&1&2&1&4&2&1&8&4&2&1\\
     v^{(3)}(-1,0)&1&-1&0&1&0&0&-1&0&0&0\\
  }.
\]
For $\OO = \{x_1,x_2,x_1^2,x_1x_2 \}$, the submatrices we need are
$$ N^* = \begin{bmatrix}
-2&3&4&-6\\3&2&9&6\\2&1&4&2 \\-1&0&1&0
\end{bmatrix}, ~ N_1 = \begin{bmatrix}
4&-6&-8&12\\9&6&27&18\\4&2&8&4\\1&0&-1&0
\end{bmatrix}, ~ N_2 = \begin{bmatrix}
-6&9&12&-18\\6&4&18&12\\2&1&4&2\\0&0&0&0
\end{bmatrix},$$
corresponding to $\OO, x_1\cdot \OO$ and $x_2 \cdot \OO$
respectively. The vector space $B$ in this example is the space of
polynomials supported in $\OO$. One can check that $N^*$ is
invertible. Using Matlab, we find the eigenvalues of $N_2 v = \lambda
N^* v$ via the command $\texttt{eig}$. The eigenvalues are $0,1,2,3$
as expected. Of course, in practice we do not know the solutions and
we cannot construct the nullspace in this way. Any basis will do,
since using another basis comes down to left multiplying $N$ and the
$N_i$ by an invertible matrix. Note that $\OO$ does not correspond to
any monomial order and it is not connected to one, so it does not
correspond to a Groebner or a border basis. 
\begin{figure} 
\centering
\input{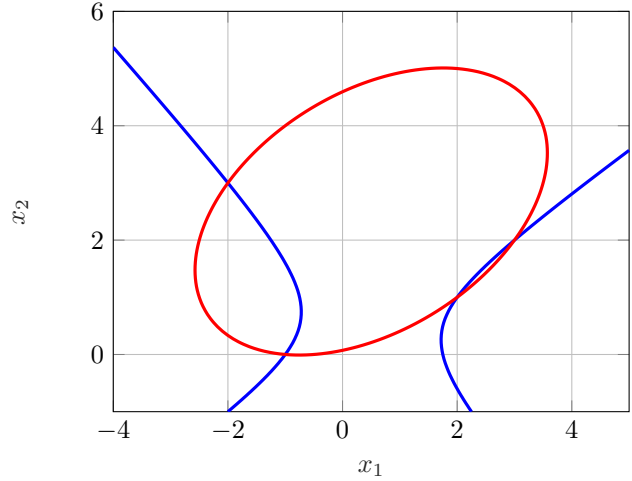}
\caption{Picture in $\R^2$ of the algebraic curves $\V(f_1)$ (\ref{bluecurve}) and $\V(f_2)$ (\ref{redcurve}) from Example \ref{ex:affine}.}
\label{fig:ex2}
\end{figure}
\end{example}

\section{Ideals defining points in $(\C^*)^n$} \label{sec:toric}
We now switch to another setting, in which we want to find the roots in the algebraic torus $(\C^*)^n$ of a set of Laurent polynomials. Denote by $$R_{x_{1}\cdots x_{n}} = \C[x_1,x_1^{-1}\ldots, x_n, x_n^{-1}] = \C[\x, \x^{-1}]$$ the localization of $R$ at $x_{1}\cdots x_{n}.$ We consider a zero-dimensional ideal $$I = \ideal{ f_1, \ldots, f_s } \subset R_{x_{1}\cdots x_{n}}$$ generated by $s$ Laurent polynomials in $n$ variables. Its localization is denoted $I^{*}= I \cdot R_{x_{1}\cdots x_{n}} \cap R$.
Hereafter, we assume that $I^{*}$ defines $\D$ solutions, counting multiplicities. These are the solutions of $I$ which are in $(\C^*)^n$.

Let $V$ be a vector space of polynomials in $R$ supported in some finite subset $\sub$ of $\NN^n$:
$$ V = \bigoplus_{\alpha \in \sub} \C \cdot x^\alpha \subset R.$$
We consider here also a map $N :  V \longrightarrow \C^\D$.

\begin{theorem} \label{prop:torus}
Assume that $\dim_{\C}R/I^{*}=\D>0$ and that $\ker(N) \subset I \cap V$.
If there is a vector subspace $W \subset V$ such that $x_i \cdot W \subset V, i = 1, \ldots, n$ and for the restriction of $N$ to $W$ we have $\rank(N_{|W}) = \D$, then for any vector subspace $B \subset W$ such that $W = B \oplus \ker(N_{|W})$, we have: 
\begin{enumerate}[(i)]
\item $N^* = N_{|B}$ is invertible, 
\item $V= B \oplus V\cap I^{*}$ and $(\ideal{ \ker (N) }:u)=I^{*}$ for any monomial
  $u\in V$.  
\item there is an isomorphism of $R$-modules $B \simeq R/I^{*}$,  
\item the maps $N_i$ given by  
\begin{alignat*}{2}
	 N_i : ~ & B \longrightarrow \C^\D,  \\  
	 &b \longrightarrow N(x_i\cdot b)  
\end{alignat*} 
for $i = 1,\ldots,n$ can be decomposed as $N_i = N_0 \circ m_{x_i}$
where $m_{x_i}: B \rightarrow B$ define the multiplications by $x_i$
in $B$ modulo $I^*$ and are commuting.
\end{enumerate}
\end{theorem}
\begin{proof} We apply Theorem \ref{thm:affine} with $I^{*}\supset I$ and $u$ any monomial of $V$ (since any monomial is invertible in $R/I^{*}$).
\end{proof}
Again, the eigenvalues $z_{ji}, j = 1, \ldots, \D$ of the $m_{x_i}$ can be computed as the generalized eigenvalues of $N_i v = \lambda N^* v$, once we have a matrix representation of $N^*$ and the $N_i, i=1, \ldots, n$.

\subsection{Constructing $N$ for square systems}

For the construction of $N$ in the toric case we rely on the famous BKK-theorem by Bernstein \cite{bernstein}, Kushnirenko \cite{kush} and Khovanskii \cite{kho} that bounds the number of solutions in the algebraic torus for a sparse, square system. To state it, we need a few definitions. More details can be found in \cite{cox2,fulton1993introduction,sturm}.
\begin{definition}[Minkowski sum]
Let $P$ and $Q$ be polytopes in $\R^n$. The \textup{Minkowski sum} of $P$ and $Q$ is 
$$P+Q = \{p+q : p \in P, q \in Q\}.$$
\end{definition}

\begin{definition}[Mixed volume]
The $n$-dimensional \textup{mixed volume} of a collection of $n$ polytopes $P_1,\ldots,P_n$ in $\R^n$, denoted $\MV(P_1,\ldots,P_n)$, is the coefficient of the monomial $\lambda_1 \lambda_2 \cdots \lambda_n$ in $\Vol_n(\sum_{i = 1}^n \lambda_i P_i)$.
\end{definition}
 
\begin{theorem}[Bernstein's Theorem] \label{ber}
Let $\I = \ideal{ f_1,\ldots,f_n } \subset R_{x_1\cdots x_n}$ define a zero-dimensional ideal and let $P_i$ be the Newton polytope of $f_i$. The number of points in $\V(\I) \cap (\C^*)^n$ is bounded above by $\MV(P_1,\ldots,P_n)$. Moreover, for generic choices of the coefficients of the $f_i$, the number of roots in $(\C^*)^n$, counting multiplicities, is exactly equal to $\MV(P_1,\ldots,P_n)$.
\end{theorem}
\begin{proof}
For sketches of the proof we refer to \cite{cox2,sturm}. For details, the reader may consult Bernstein's original paper \cite{bernstein}. A proof based on homotopy continuation is given in \cite{hustu}.
\end{proof}
The type of genericity we assume in this section is that the number of solutions of $\I$ in $(\C^*)^n$, counting multiplicities, is exactly $\MV(P_1,\ldots,P_n)$. Let $f_0$ be a generic linear polynomial and let $v \in \R^n$ be a generic, small $n$-tuple. We consider the resultant map
\begin{eqnarray*}
	 M_0 : \quad  V_{0}\times V_1 \times \cdots \times V_n& \longrightarrow& V  \\
	 (q_{0},q_1, \ldots, q_n) &\longmapsto & q_{0}f_{0}+ q_1f_1+ \cdots + q_nf_n.
\end{eqnarray*}
where $V_i = \bigoplus_{\alpha \in \sub_i} \C \cdot x^\alpha$, $\sub_i = (P_0 + \ldots + \hat{P}_i + \ldots + P_n + v) \cap \Z^n$ (the notation $\hat{P}_i$ means that this term is left out) and $V = \bigoplus_{\alpha \in \sub} \C \cdot x^\alpha$, $\sub = (\sum_{i=0}^n P_i + v) \cap \Z^n$. We can select a square submatrix $M'$ of this map, so that $\det(M')$ is a nontrivial multiple of the toric resultant of $f_0,f_1,\ldots,f_n$ \cite{emir1,cox2}. We set $W = \{ f \in V : x_i \cdot f \in V, i = 1,\ldots,n \}$. As for the Macaulay resultant matrix, we write
$$ M' = \begin{bmatrix}
M_{00} & M_{01}\\ M_{10} & M_{11}
\end{bmatrix}
$$
where the rows and columns of $M_{00}$ are indexed by a set $\OO_0 \subset W$ of monomials which is a basis of $R/I$ and $M_{11}$ is invertible. Denoting as in Section \ref{sec:affine} by $\tilde{M}$ the right block column of $M'$ and by $N$ its left null space, we have again that $\ker(N) = \im(\tilde{M}) = \im(M) = I\cap V$ with $M = (M_0)_{|V_1 \times \cdots \times V_n}$. Since $\OO_0 \subset W$, $N_{|W}$ is surjective and we can apply Theorem \ref{prop:torus}. Algorithm \ref{alg:toric} only differs from Algorithm \ref{alg:affine} by the construction of $M$.

\begin{algorithm}[ht!]
\caption{Computes the algebra structure of $R/I^*$ (generic, sparse case)}\label{alg:toric}
\begin{algorithmic}[1]
\Procedure{AlgebraStructure}{$f_1,\ldots,f_n$}
\State $M \gets \textup{the toric resultant map on $V_1 \times \cdots \times V_n$}$ 
\State $N \gets \nulll(M^{\top})^\top$ \label{nullspace-sparse}
\State $N_{|W} \gets \textup{columns of $N$ corresponding to monomials $x^\alpha$ such that $x^\alpha \in W$}$ 
\State Apply Algorithm \ref{alg:affine} from step 5 onward.
\EndProcedure
\end{algorithmic}
\end{algorithm}

\section{Ideals defining points in $\PP^n$} \label{sec:proj}

Suppose we are interested in finding all projective roots of a system of homogeneous equations. Denote $\Rh = \C[x_0,x_1,\ldots,x_n]$ and let $\I = \ideal{ f_1, \ldots, f_s } \subset \Rh$ be a zero-dimensional ideal generated by $s$ homogeneous polynomials in $n+1$ variables with $\D < \infty$ solutions in $\PP^n$, counting multiplicities.
For $d\in \NN$, let $V = \Rh_d$ be the degree $d$ part of $\Rh$ and suppose we have a map $N : V \rightarrow \C^\D$ such that $\ker(N) \subset \I_d = \I\cap V$.
We also assume that there exists $h\in S_{1}$ such that the map
\begin{alignat*}{2}
	 {N_{h}} : ~ & \Rh_{d-1} \longrightarrow \C^\D,  \\
	 &f \longrightarrow N(h \cdot f) 
\end{alignat*}
is surjective.
Let \begin{alignat*}{2}
	 {N_{i}} : ~ & \Rh_{d-1} \longrightarrow \C^\D,  \\
	 &f \longrightarrow N(x_{i} \cdot f).
\end{alignat*}
Then $N_{h}=\sum_{i=1}^{n} h_{i} N_{i}$ where $h=h_{0}x_{0}+ \cdots+ h_{n} x_{n}$. Without loss of generality, we assume $h_0 \neq 0$.

Let $R = \C[y_1,\ldots,y_n]$ be the ring of polynomials in $n$ variables. We have homogenization isomorphisms 
\begin{alignat*}{2}
	 \sigma_d : ~ & R_{\leq d} \longrightarrow \Rh_d,  \\
	 &f \longrightarrow h^d f \left ( \frac{x_1}{h},\ldots, \frac{x_n}{h} \right)
\end{alignat*}
for every $d \in \NN$. The inverse dehomogenization map in degree $d$ is given by 
\begin{alignat*}{2}
	 \sigma_{d}^{-1} : ~ & \Rh_{d} \longrightarrow R_{\leq d},  \\
	 &f \longrightarrow f \left (\frac{1-\sum_{i=1}^n h_i y_i}{h_0},y_1,\ldots, y_n \right).
       \end{alignat*}
Its definition is independent of the degree $d$, so that we can omit $d$ and denote it $\sigma^{-1}$.
The ideal $\tilde{I} =  \ideal{ \sigma^{-1}(f_1), \ldots, \sigma^{-1}(f_n) }$ has $\delta$ solutions in $\C^n$, counting multiplicities. Let $\tilde{V} = R_{\leq d}$ and $\tilde{W} = R_{\leq d - 1}$. The map $\tilde{N}: \tilde{V} \rightarrow \C^\D$ given by $\tilde{N} = N \circ \sigma_d$ is surjective and $\ker(\tilde{N}) \subset \tilde{I} \cap \tilde{V}$. Also, $y_i \cdot \tilde{W} \subset \tilde{V}$, $i = 1,\ldots,n$. For $f \in R_{\leq d - 1}$, $\sigma_d(f) = h \cdot \sigma_{d-1}(f)$. Therefore $\tilde{N}(R_{\leq d-1}) = N(h \cdot \Rh_{d-1})$ and $\tilde{N}_{|\tilde{W}} = {N_{h}} \circ \sigma_{d-1}$ is surjective. 

\begin{theorem} \label{prop:projective}
Let $B \subset \Rh_{d-1}$ be any subspace such that $\Rh_{d-1} = B \oplus \ker({N_{h}})$. Under the above assumptions,
\begin{enumerate}[(i)]
\item $N^{*} = {(N_{h})}_{|B}$ is invertible,
\item there is an isomorphism of $\C\left [ \frac{x_0}{h}, \cdots, \frac{x_n}{h} \right ]$-modules $h \cdot B \simeq S_{d}/\I_{d} $,
\item $\Rh_{k} = h^{k-d+1}\cdot B \oplus \I_{k}$ for $k\geq d$ and
  $I=(\ideal{ \ker (N) }:h^{*})$. 
\item the maps $N_i$ given by \begin{alignat*}{2}
	 N_i: ~ & B \longrightarrow \C^\D,  \\ 
	 & b \longrightarrow N(x_i\cdot b) 
\end{alignat*}
for $i = 0,\ldots,n$ can be decomposed as $N_i = N^{*} \circ m_{x_i}$, 
where $m_{x_i}$ represent the multiplications by $x_i/h$ in $h\cdot B$
modulo $\I_{d}$ and are commuting. 
\end{enumerate}
\end{theorem}
\begin{proof}
  We apply Theorem \ref{thm:affine} to $\tilde{N}$ and $\tilde{W}=R_{\leq d-{1}}\ni 1$.
  Let $B$ be a supplementary space of  $\ker({N_{h}})$ in $\Rh_{d-1}$. Then
  $\tilde{B}=\sigma^{-1}(B)$ is a supplementary  vector space of $\ker(\tilde{N}_{|\tilde{W}})$ in $\tilde{W}$.
\begin{enumerate}[(i)]
\item $(N_{h})_{|B} = \tilde{N}_{|\tilde{B}} \circ \sigma^{-1}$ is
  invertible  since $\tilde{N}_{|\tilde{B}}$ is invertible. Note that
  $\sigma^{-1}(h \cdot b) = \sigma^{-1}(b) \in R_{\leq d-1}$ since
  $\sigma^{-1}(h) = 1$.
\item Theorem \ref{thm:affine} gives $\tilde{B} \simeq R/\tilde{I}$. Applying $\sigma_{d}$ gives $h \cdot B \simeq \Rh_{d}/\I_{d}$. 
\item Theorem \ref{thm:affine} implies that
  $\tilde{V} = \tilde{B} \oplus \tilde{V} \cap \tilde{\I}$. More
  generally, we have
  $\tilde{R}_{\le k} = \tilde{B} \oplus \tilde{R}_{\le k} \cap
  \tilde{\I}$ for $k\ge d$. By applying $\sigma_{k}$ for $k\geq d$, we
  have $S_{k} = h^{k-d+1} B \oplus I_{k}$.  From \ref{thm:affine}, we also
  have $\tilde{I}=\ideal{ \ker (\tilde{N}) }$. By applying $\sigma_{k}$ for
  $k\in \mathbb{N}$, we deduce that $I=(\ideal{ \ker (N) }:h^{*})$, which proves the third
  point.

\item  Let $m_{y_i}$ be the maps from Theorem \ref{thm:affine}. Consider the induced maps
  $$
  m_{x_i} = \sigma_{d} \circ m_{y_i} \circ \sigma^{-1}, i=1,\ldots,n
  $$ and
  $$
  m_{x_0} =  \sigma_{d} \circ \left ( \frac{1 - \sum_{i=1}^n h_i y_i}{h_0} \right) (\m) \circ \sigma^{-1},
  $$
By definition, for $i = 1,\ldots, n$ and $b\in B$, we have $N(x_i \cdot b) = \tilde{N}(\sigma^{-1}(x_i \cdot b)) = \tilde{N}(y_i \cdot  \sigma^{-1}(b))$. By Theorem \ref{thm:affine} this can be written as $ \tilde{N}(y_i \cdot  \sigma^{-1}(b)) = (\tilde{N}_{|\tilde{B}} \circ m_{y_i}) (\sigma^{-1}(b)) = (N^{*} \circ \sigma_{d} \circ m_{y_i}) (\sigma^{-1}(b))$. And since $\sigma_{d} \circ m_{y_i} = m_{x_i} \circ \sigma_{d}$ we get $N_i(b) = (N^{*} \circ m_{x_i})(h \cdot b)$. Analogously, using linearity, for $N_0$ we have 
$$
N(x_0 \cdot b) = \tilde{N} \left ( \frac{1-\sum_{i=1}^n h_iy_i}{h_0} \cdot \sigma^{-1}(b) \right ) = (N^{*} \circ m_{x_0}) (h \cdot b).
$$
We now show that $m_{x_i}$ represents the multiplication by $x_i/h$ in $h\cdot B\subset h\cdot \Rh_{d-1}$ modulo $\I_{d}$.
For $b \in B$, let $\sigma^{-1}(h\cdot b) = \sigma^{-1}(b)  = \tilde{b}\in \tilde{B}$ and $m_{y_i}(\tilde{b}) = y_i \cdot \tilde{b} - p$ with $p \in \tilde{I}$. Then for $i = 1,\ldots,n$,
$$
m_{x_i}(h\cdot b) = \sigma_{d}(y_i \cdot \tilde{b}-p) = {x_i} \cdot \sigma_{d-1}(\tilde{b}) - \sigma_d(p) = {x_i} \cdot b \mod \I_{d}.
$$
For $m_{x_0}$, the result follows from $\sigma_{1} \left ( \frac{1-\sum_{i=1}^n h_i y_i}{h_0} \right ) = x_0$.
\end{enumerate}
\end{proof}

It follows that once we have a matrix representation of $N^*$ and of the $N_i$, we have that $m_{x_i} = (N^*)^{-1} N_i$ and the matrices $(N^*)^{-1} N_i$ commute, so that for an eigenvalue $\lambda_i = \frac{z_{ji}}{h(z_{j})}$ of $m_{x_i}$ and $\lambda_k = \frac{z_{jk}}{h(z_{j})}$ of $m_{x_k}$ with common eigenvector $v$:
$$ \lambda_k(N^*)^{-1}N_i v = \lambda_k \lambda_i v= \lambda_i (N^*)^{-1} N_k v.$$
Left multiplication by $N^*$ gives $\lambda_k N_i v = \lambda_i N_k v$ and the generalized eigenvalues of $N_i v = \lambda N_k v$ are the fractions $z_{ji}/z_{jk}$. This means that we do not need to construct $N^*$ to find the projective coordinates of the solutions, as long as we have $N_i, i = 0,\ldots, n$ and a generic linear combination of the $N_i$ is invertible. \\

The following proposition shows that the hypotheses of Theorem \ref{prop:projective} can be fulfilled  for $d$ greater than or equal to the regularity and provides a new criterion for detecting the $d$-regularity of a projective zero-dimensional ideal. We recall that the regularity $\reg(I)$ of an ideal $I$ is $\min(d_{i,j}-i)$ where $d_{i,j}$ are the degrees of generators of the $i^{\rm th}$-syzygy module in a minimal resolution of $I$ (see \cite{eisenbud_geometry_2005}). An ideal is $d$-regular if $d\geq \reg(I)$.
\begin{proposition}\label{prop:regularity}
Let $I$ be a homogeneous ideal with $\D < \infty$ solutions in $\PP^n$, counting multiplicities. The following statements are equivalent:
\begin{enumerate}[(i)]
\item There is a linear map $N : \Rh_{d} \rightarrow \C^\D$ with $\ker(N) \subset I \cap \Rh_d$ and $N_h: S_{d-1} \rightarrow \C^\D$ given by $N_h(f) = N(h \cdot f)$ is surjective for generic $h$, \label{linmaps}
\item $I$ is $d$-regular. \label{rhoreg}
\end{enumerate}
\end{proposition}
\begin{proof}
  (\ref{linmaps}) $\Rightarrow$ (\ref{rhoreg}). From Proposition \ref{prop:projective} it follows that we can find $B \subset S_{d-1}$ such that $\Rh_d = h \cdot B \oplus I_d$. Therefore $\Rh_{d} = \ideal{ I , h}_{d}$.  Denote $(I : h) = \{f \in \Rh : fh \in I \}$.
 Let $f \in (I:h)_{d}$. Then $f \equiv h\cdot b \mod I_{d}$ with $b\in B$ and $h^{2}\cdot b \in I_{d+1}$. As we have $\Rh_{d+1}= h^{2}\cdot B\oplus I_{d+1}$, we deduce that $b=0$ and $(I : h)_{d}=I_{d}$. 
By \cite{bayer1987criterion}[Theorem 1.10], $I$ is $d$-regular.

 (\ref{rhoreg}) $\Rightarrow$ (\ref{linmaps}).
Assume that $I$ is $d$-regular. Let $\D = \dim_\C( \Rh_{d}/I_{d})$. By \cite{eisenbud_geometry_2005}[Theorem 4.2 (3)], $d$ is greater or equal to the regularity index of the Hilbert function. Therefore $\delta$ is the value of the constant Hilbert polynomial, that is, the number of solutions in $\V(I)$ counting multiplicities. 
Consider a basis $\{ \eta_1, \ldots, \eta_\D \}$ of $I_{d}^{\perp} \subset \Rh_{d}^{*}$. Define 
\begin{alignat*}{2}
	 N: ~ & \Rh_d \longrightarrow \C^\D,  \\
	 & f \longrightarrow (\eta_i(f))_{1\leq i \leq \D}.
\end{alignat*}
By construction, $\ker(N) = I_d$. By $d$-regularity,  $\ideal{ I,h}_d=S_{d}$  for a generic $h\in \Rh_{1}$ (see \cite{bayer1987criterion}[Theorem 1.10]). For any $f \in \Rh_d$, we can write $f = \tilde{f} + hg$ with $\tilde{f} \in I_d$ and $g \in \Rh_{d-1}$. Therefore $N(f) = N(hg)= N_{h}(g)$ and $N(\Rh_d) = N_h(\Rh_{d-1})$.
\end{proof}
\subsection{Constructing $N$ for square systems}

From the discussion above, $\tilde{N}$ and $N$ have the same matrix representation. 
We show how the maps $N, N^*, N_i$ can be constructed from the null space of the resultant map $M$, used in the affine, dense case: $\rho = \sum_{i=1}^n d_i -(n-1)$. For generic $h = h_0x_0 + \ldots + h_nx_n$, $h_0 \neq 0$, a change of coordinates given by 
$$ \begin{bmatrix}
\hat{x}_0 \\ \hat{x}_1\\ \vdots \\ \hat{x}_n \end{bmatrix} =  \begin{bmatrix}
h_0 & h_1 & \cdots & h_n \\
&1 \\
& & \ddots \\
& & & 1
\end{bmatrix} \begin{bmatrix}
x_0 \\ x_1 \\ \vdots \\ x_n
\end{bmatrix}$$
does not alter the rank of the resultant map $M$ and the resulting system has $\D$ affine solutions in $\PP^n \backslash \{\hat{x}_0 = 0 \}$. In the notation from this section, the associated null space map is $\tilde{N} = N \circ \sigma_\rho$ and $\tilde{N}_{|\tilde{W}} = N_h \circ \sigma_{\rho-1}$ with $\tilde{W} = R_{\leq \rho-1}$ and these maps have all the good properties by the results of Section \ref{sec:affine}. We obtain Algorithm \ref{alg:proj}, where the `homogeneous Macaulay matrix' is the matrix from Algorithm \ref{alg:affine} with columns corresponding to homogeneous polynomials and rows indexed by monomials of degree $\rho$. Note that Algorithm \ref{alg:affine} is  equivalent to Algorithm \ref{alg:proj} when we use $h = x_0$.

\begin{algorithm}[h!]
\caption{Computes the algebra structure of $S_{\rho}/I_{\rho}$}\label{alg:proj}
\begin{algorithmic}[1]
\Procedure{AlgebraStructure}{$f_1,\ldots,f_n$}
\State $M \gets \textup{homogeneous Macaulay matrix of degree } \rho = \sum_{i=1}^n d_i - (n-1)$ \label{constmac}
\State $N \gets \nulll(M^\top)^\top$ \label{nullspace}
\State $\OO_{\rho-1} \gets \textup{monomials of degree $\rho-1$}$
\For{$i=0,\ldots,n$}
\State $N_{|W_i} \gets \textup{columns of $N$ corresponding to $x_i \cdot \OO_{\rho-1}$}$ 
\EndFor
\State $h \gets \textup{generic linear form}$
\State $N_h \gets h(N_{|W_0},\ldots,N_{|W_n})$ 
\State $N^* \gets \textup{columns of $N_h$ corresponding to an invertible submatrix}$
\For{$i=0,\ldots,n$}
\State $N_i \gets \textup{columns of $N_{|W_i}$ corresponding to the columns of $N^*$}$
\State $m_{x_i} \gets (N^*)^{-1}N_i$
\EndFor
\State \textbf{return} $m_{x_0},\ldots, m_{x_n}$
\EndProcedure
\end{algorithmic}
\end{algorithm}

\begin{example}
We give an example of a zero-dimensional system of homogeneous equations coming from an affine system with a solution at infinity. Consider the equations $f_1 = 2x_1^2 + 5x_1x_2 + 3x_2^2+3x_1-2 = 0$ and $f_2 = -2 +x_1+x_2 = 0$. After homogenizing we get 
\begin{eqnarray*}
f_1^h &=& 2x_1^2 + 5x_1x_2 + 3x_2^2+3x_0x_1-2x_0^2 = 0,\\
f_2^h &=& -2x_0+ x_1+x_2= 0,
\end{eqnarray*}
with solutions $z_1 = (0,1,-1), z_2 = (1,-10,12) \in \PP^2$. Since $f_2^h$ is linear, the system could be solved fairly easily by using substitution, but we use this example nonetheless because the matrices involved are not too large and it illustrates the algorithm nicely. We have $\rho = 2$ and we set
$$ w^{(2)}(x_0,x_1,x_2) = \begin{bmatrix}
x_0^2 & x_0x_1 & x_0x_2 & x_1^2& x_1x_2 & x_2^2
\end{bmatrix}.$$
We get a null space matrix
\[
  N=\kbordermatrix{%
     &x_0^2 & x_0x_1 & x_0x_2 & x_1^2& x_1x_2 & x_2^2\\
     w^{(2)}(0,1,-1)& 0 & 0 & 0 &1&-1&1\\
     w^{(2)}(0,-10,12)&1&-10&12&100&-120&144\\
  }.
\]
Note that we cannot apply Algorithm \ref{alg:affine}, since after dehomogenizing by $x_0=1$, there is no invertible submatrix of the degree $1$ part of $N$. The $N_{|W_i}$ are 
$$ N_{|W_0} = \begin{bmatrix}
0&0&0\\1&-10&12
\end{bmatrix}, ~ N_{|W_1}=\begin{bmatrix}
0&1&-1\\-10&100&-120
\end{bmatrix}, ~ N_{|W_2} = \begin{bmatrix}
0& -1 & 1\\12&-120&144
\end{bmatrix}.$$
A generic linear combination of the first 2 columns of these matrices is invertible. We set $N_i$ to be the first two colums of $N_{|W_i}$. We find that the pencil $N_1-\lambda N_0$ has eigenvalues $\infty, -10$, which corresponds to the $x_1$-values of the solutions in the affine chart $x_0=1$. We computed this without constructing $N^*$. For a generic linear form $h$, set $N^* = h(N_0,N_1,N_2)$. The eigenvalues of $(N^*)^{-1}N_i$ are the values of the $i$-th coordinate function at the solutions evaluated at $h(x_0,x_1,x_2) = 1$.
\end{example}

\section{Ideals defining points in $\PP^{n_1} \times \cdots \times \PP^{n_k}$} \label{sec:multihom}
We want to find all roots in $\PP^{n_1} \times \cdots \times \PP^{n_k}$ of a system of multihomogeneous equations. Denote $\Rh = \C[x_{10}, \ldots, x_{1n_1}, \ldots, x_{k0}, \ldots, x_{kn_k}]$ and let $\I = \ideal{ f_1, \ldots, f_s } \subset \Rh$ be an ideal defined by multihomogeneous equations. Here, we take $x_{i0}, \ldots, x_{in_i}$ to be the projective coordinates on the $i$-th factor $\PP^{n_i}$ in $\PP^{n_1} \times \cdots \times \PP^{n_k}$. We assume that $\I$ has $\D < \infty$ solutions in $\PP^{n_1} \times \cdots \times \PP^{n_k}$. By $\Rh_\rho$, $\rho \in \NN^k$ we denote the multidegree $\rho$ part of $\Rh$. That is, $\Rh_\rho$ consists of the elements of $\Rh$ of degree $\rho_i$ in $x_{ij}, j = 0,\ldots,n_i$. Let $V = S_\rho$ and suppose we have $N: V \rightarrow \C^\D$ surjective and $\ker(N) \subset I_\rho = I \cap V$. Denoting $\1 = \sum_{i=1}^k e_i$, we assume that there are linear forms $h_i = h_{i0}x_{i0} + \ldots, h_{in_i}x_{in_i} \in S_{e_i}, i=1,\ldots,k$ such that 
\begin{alignat*}{2}
	 N_{h} : ~ & \Rh_{\rho-\1} \longrightarrow \C^\D,  \\
	 &f \longrightarrow N(h_1\cdots h_k \cdot f) 
\end{alignat*}
is surjective.
We proceed as in the projective case by defining (de)-homogenization isomorphisms.

Let $R = \C[y_{11},\ldots,y_{1n_1},\ldots,y_{k1},\ldots,y_{kn_k}]$ be the ring of polynomials in $n=\sum_{i=1}^k n_i$ variables. We have homogenization isomorphisms 
\begin{alignat*}{2}
	 \sigma_\rho : ~ & R_{\leq \rho} \longrightarrow \Rh_\rho,  \\
	 &f \longrightarrow h_1^{\rho_1} \cdots h_k^{\rho_k} f \left ( \frac{x_{11}}{h_1},\ldots,\frac{x_{1n_1}}{h_1} \ldots,  \frac{x_{k1}}{h_k},\ldots,\frac{x_{kn_k}}{h_k} \right)
\end{alignat*}
for every $\rho \in \NN^k$. The inverse dehomogenization map is given by 
\begin{alignat*}{2}
	 \sigma^{-1} : ~ & \Rh \longrightarrow R,  \\
	 &f \longrightarrow f \left (\frac{1-\sum_{i=1}^{n_1} h_{1i} y_{1i}}{h_{10}},y_{11},\ldots, y_{1n_1}, \ldots,  \frac{1-\sum_{i=1}^{n_k} h_{ki} y_{ki}}{h_{k0}},y_{k1},\ldots, y_{kn_k} \right).
\end{alignat*}
The ideal $\tilde{I} =  \ideal{ \sigma^{-1}(f_1^h), \ldots, \sigma^{-1}(f_n^h) }$ has $\delta$ solutions in $\C^n$, counting multiplicities. Let $\tilde{V} = R_{\leq \rho}$ and $\tilde{W} = R_{\leq \rho - 1}$. The map $\tilde{N}: \tilde{V} \rightarrow \C^\D$ given by $\tilde{N} = N \circ \sigma_\rho$ is surjective and $\ker(\tilde{N}) \subset \tilde{I} \cap \tilde{V}$. Also, $y_{ij} \cdot \tilde{W} \subset \tilde{V}$, $i = 1,\ldots,k, j = 1,\ldots,n_i$. For $f \in R_{\leq \rho - \1}$, $\sigma_\rho(f) = h_1 \cdots h_k \cdot \sigma_{\rho-\1}(f)$. Therefore $\tilde{N}(R_{\leq \rho-\1}) = N(h_1\cdots h_k \cdot \Rh_{\rho-\1})$ and $\tilde{N}_{|\tilde{W}} = N' \circ \sigma_{\rho-\1}$ is surjective. 

\begin{theorem} \label{prop:multihom}
Let $B \subset \Rh_{\rho-\1}$ be any subspace such that $\Rh_{\rho-\1} = B \oplus \ker(N_{h})$. Under the above assumptions,
\begin{enumerate}[(i)]
\item $N^* = (N_{h})_{|B}$ is invertible,
\item there is an isomorphism of $\C\left [
    \frac{x_{11}}{h_1},\ldots,\frac{x_{1n_1}}{h_1} \ldots,
    \frac{x_{k1}}{h_k},\ldots,\frac{x_{kn_k}}{h_k} \right ]$-modules $
  h_1 \cdots h_k \cdot B \simeq S_{\rho}/\I_{\rho} $,
\item $V = h_1 \cdots h_k \cdot B \oplus V \cap \I$ and $I=(\ideal{ \ker (N)}:
  (h_{1}\cdots h_{k})^{*})$.  
\item the maps $N_{ij}$ given by \begin{alignat*}{2}
	 N_{ij}: ~ & B \longrightarrow \C^\D,  \\  
	 &  b \longrightarrow N(h_1 \cdots \hat{h}_i \cdots h_k \cdot x_{ij}\cdot b)   
\end{alignat*} 
for $i = 1,\ldots,k,j=0,\ldots,n_i$ can be decomposed as $N_{ij} = N^*
\circ m_{x_{ij}}$, where $m_{x_{ij}}$ represent the multiplications by
$x_{ij}/h_i$ in $h_1 \cdots h_k \cdot B$ modulo $\I_{\rho}$ and are commuting.  
\end{enumerate}
\end{theorem}
\begin{proof}
All statements follow from Theorem \ref{thm:affine} as in the proof of Theorem \ref{prop:projective}.
\end{proof}

\subsection{Constructing $N$ for square systems}
We show that the maps $N, N^*, N_{ij}$ can be constructed from the null space of the toric resultant map $M$ as defined for the affine sparse case. Let $I = \ideal{ f_1,\ldots, f_n }$ be defined by $n = n_1 + \ldots + n_k$ multihomogeneous polynomials of degrees $d_i \in \NN^k$. A change of projective coordinates within each factor $\PP^{n_i}$ does not alter the rank of the resultant map. Take
$$ \begin{bmatrix} \hat{x}_{1} \\ \hat{x}_{2} \\  \vdots \\ \hat{x}_{k}\end{bmatrix} = \begin{bmatrix}
H_1 \\ & H_2 \\ && \ddots \\ &&& H_k
\end{bmatrix}  \begin{bmatrix} x_1\\ x_2 \\ \vdots \\ x_k\end{bmatrix}, \quad H_i = \begin{bmatrix}
h_{i0} & h_{i1} & \ldots & h_{in_i}\\
& 1 \\ & & \ddots \\ & & & 1
\end{bmatrix}$$
where $\hat{x}_i, x_i$ are short for $(\hat{x}_{i0}, \ldots, \hat{x}_{in_i})^\top$  and $(x_{i0},\ldots, x_{in_i})^\top$ respectively. Using the notation in this chapter, the resulting ideal  after dehomogenization w.r.t. the $\hat{x}_{i0}$ is $\tilde{I} = \ideal{
\tilde{f}_1,\ldots, \tilde{f}_n } \subset R = \C[\hat{x}_1,\ldots, \hat{x}_k]$ with $\D = \MV(P_1,\ldots, P_n)$ solutions in $\C^n$, counting multiplicities. We may assume that all $\D$ solutions lie in $(\C^*)^n$, since we can apply another generic block diagonal change of coordinates. Next, we consider a generic polynomial $\tilde{f}_0 \in R_{\leq \1}$, so $d_0=\1$. We denote $\rho = \sum_{i=0}^n d_i - \1$ and $\rho_i = \rho-d_i, i = 0,\ldots,n$ and we consider the resultant map\begin{eqnarray*}
	 \tilde{M}_0 : \quad  \tilde{V}_{0}\times \tilde{V}_1 \times \cdots \times \tilde{V}_n& \longrightarrow& \tilde{V}  \\
	 (q_{0},q_1, \ldots, q_n) &\longmapsto & q_{0}\tilde{f}_{0}+ q_1\tilde{f}_1+ \cdots + q_n\tilde{f}_n.
\end{eqnarray*}
with $\tilde{V}_i = R_{\leq \rho_i}$ and $\tilde{V} = R_{\leq \rho}$. Note that this corresponds to the $V_i, V$ in the affine, sparse case, where we take a vector $v = \epsilon(-1,\ldots,-1) \in \R^n$ with $\epsilon > 0$, small. We denote $\tilde{W} = R_{\leq \rho - \1}$. By the discussion in Section \ref{sec:toric}, the null space map $\tilde{N}$ associated to $(\tilde{M_0})_{|\tilde{V}_1 \times \cdots \times \tilde{V}_n}$ and the map $\tilde{N}_{|\tilde{W}}$ have all the good properties. By construction, $\tilde{N} = N \circ \sigma_\rho$ and $\tilde{N}_{|\tilde{W}} = N_h \circ \sigma_{\rho - \1}$ where $N$ is the null space map associated to \begin{eqnarray*}
	 M : \quad  V_{1} \times \cdots \times {V}_n& \longrightarrow& {V}  \\
	 (q_1, \ldots, q_n) &\longmapsto &  q_1f_1+ \cdots + q_nf_n,
\end{eqnarray*}
with $V_i = \Rh_{\rho_i}$ and $V = \Rh_\rho$.\\

In Algorithm \ref{alg:multihom} we use the notation $\textup{vec}: \Rh_\rho \rightarrow \C^{m_\rho}$, where $m_\rho = \dim_\C(V)$ is the number of rows of the matrix $M$, for the map that sends a multihomogeneous polynomial of degree $\rho$ to its column vector representation corresponding to the monomials in the support of $M$.
\begin{algorithm}[ht!]
\caption{Computes the algebra structure of $S_{\rho}/I_{\rho}$}\label{alg:multihom}
\begin{algorithmic}[1]
\Procedure{AlgebraStructure}{$f_1,\ldots,f_n$}
\State $M \gets \textup{the multihomogeneous Macaulay matrix of degree $\rho$} $ \label{constmac-multihom}
\State $N \gets \nulll(M)^\top$ \label{nullspace-multihom}
\State $\OO_{\rho-\1} \gets \textup{monomials of degree $\rho-\1$}$
\For{$i=1,\ldots,k$}
\State $h_i \gets \textup{generic linear form of degree $e_i$}$ 
\EndFor
\State $K \gets \textup{empty matrix}$
\For{$m \in \OO_{\rho-\1}$}
\State $K \gets \begin{bmatrix}
K & \textup{vec}(h_1\cdots h_k \cdot m)
\end{bmatrix} $
\EndFor
\State $N_h \gets NK$
\State $N^* \gets \textup{columns of $N_h$ corresponding to an invertible submatrix}$
\State $\OO \gets \textup{monomials in $\OO_{\rho-\1}$ corresponding to the columns of $N^*$}$
\For{$i=1,\ldots,k$}
\For{$j= 0,\ldots,n_i$}
\State $K_{ij} \gets \textup{empty matrix}$
\For{$m \in \OO$}
\State $K_{ij} \gets \begin{bmatrix}
K_{ij} & \textup{vec}(h_1\cdots \hat{h}_i \cdots h_k \cdot x_{ij} \cdot m)
\end{bmatrix}$
\EndFor
\State $N_{ij} \gets N K_{ij}$
\State $m_{x_{ij}} = (N^*)^{-1}N_{ij}$
\EndFor
\EndFor
\State \textbf{return} $m_{x_{ij}},i=1,\ldots,k, j= 0,\ldots n_i$
\EndProcedure
\end{algorithmic}
\end{algorithm}

\begin{example}
We work out an example in $\PP^1 \times \PP^1$. We start with the affine equations $f_1 = 2-x_1+2x_2+2x_1x_2 = 0$ and $f_2 = 4 - 2x_1+x_2+4x_1x_2=0$. Homogenizing we get 
\begin{eqnarray*}
f_1^h &=& 2x_{10}x_{20} - x_{20}x_{11} +2x_{10}x_{21} + 2 x_{11}x_{21},\\
f_2^h &=& 4x_{10}x_{20} - 2x_{20}x_{11} +x_{10}x_{21} + 4 x_{11}x_{21}
\end{eqnarray*}
Using the coordinates $(x_{10},x_{11},x_{20},x_{21})$ on $\PP^1\times \PP^1$, the solutions are $z_1 = (1,2,1,0), \newline z_2 = (0,1,1,1/2)$. Note that $z_2$ corresponds to a solution `at infinity', in the sense that it lies on the torus invariant divisor $x_{10}=0$. A null space matrix is 
\setcounter{MaxMatrixCols}{20}
$$N = \begin{bmatrix}
1 & 2 & 4 & 8 & 0 & 0 & 0 & 0     & 0 & 0 & 0 & 0     & 0 & 0 & 0 & 0 \\
0 & 0 & 0 & 1 & 0 & 0 & 0 & 1 / 2 & 0 & 0 & 0 & 1 / 4 & 0 & 0 & 0 & 1/8
\end{bmatrix}$$
where the first row corresponds to $z_1$ and the second to $z_2$ and the columns correspond to the monomials 
\begin{center}
 $x_{10}^3x_{20}^3, x_{10}^2x_{11}x_{20}^3, x_{10}x_{11}^2x_{20}^3,x_{11}^3x_{20}^3, x_{10}^3x_{20}^2x_{21}, x_{10}^2x_{11}x_{20}^2x_{21}, x_{10}x_{11}^2x_{20}^2x_{21},x_{11}^3x_{20}^2x_{21}, \newline  x_{10}^3x_{20}x_{21}^2, x_{10}^2x_{11}x_{20}x_{21}^2, x_{10}x_{11}^2x_{20}x_{21}^2,x_{11}^3x_{20}x_{21}^2, x_{10}^3x_{21}^3, x_{10}^2x_{11}x_{21}^3, x_{10}x_{11}^2x_{21}^3,x_{11}^3x_{21}^3 $
\end{center}
in that order. In this example, we can take $h_1 = x_{10} + x_{11}$, $h_2 = x_{20}+x_{21}$. For $\OO = \{x_{11}^2x_{21}^2,x_{10}x_{11}x_{20}^2\}$, with respect to the same set of monomials, we find 
\begin{eqnarray*}
v_1 = \textup{vec}(h_1\cdot h_2\cdot x_{11}^2x_{21}^2) &=& \begin{bmatrix}
0&0&0&0&0&0&0&0&0&0&1&1&0&0&1&1
\end{bmatrix}^\top,\\
v_2 = \textup{vec}(h_1\cdot h_2 \cdot x_{10}x_{11}x_{20}^2) &=&
\begin{bmatrix}
0&1&1&0&0&1&1&0&0&0&0&0&0&0&0&0
\end{bmatrix}^\top
\end{eqnarray*}
and with $\tilde{K} = \begin{bmatrix}
v_1 & v_2 \end{bmatrix}$ we find
$$N^* = N \tilde{K} = \begin{bmatrix}
0&6\\3/8&0
\end{bmatrix}$$
invertible. Then, $$K_{10} = \begin{bmatrix}\textup{vec}(h_2 \cdot x_{10} \cdot x_{11}^2x_{21}^2) & \textup{vec}(h_2 \cdot x_{10} \cdot x_{10}x_{11}x_{20}^2) \end{bmatrix} = \begin{bmatrix}
e_{11}+e_{15} & e_{2}+e_{6}
\end{bmatrix}$$
which gives $N_{10} = \begin{bmatrix}
0&2\\0&0
\end{bmatrix}$. Analogously, 
$$K_{11} = \begin{bmatrix}\textup{vec}(h_2 \cdot x_{11} \cdot x_{11}^2x_{21}^2) & \textup{vec}(h_2 \cdot x_{11} \cdot x_{10}x_{11}x_{20}^2) \end{bmatrix} = \begin{bmatrix}
e_{12}+e_{16} & e_{3}+e_{7}
\end{bmatrix}$$
which gives $N_{11} = \begin{bmatrix}
0&4\\3/8&0
\end{bmatrix}$. We find that the eigenvalues of $(N^*)^{-1}N_{10}$ are 0 and 1/3, corresponding to $\frac{x_{10}}{h_1}(z_i)$. For the generalized eigenvalue problem defined by $N_{11} - \lambda N_{10}$, we find eigenvalues 2 and $\infty$, corresponding to the $x_1$-coordinates of the affine solutions of the original system of equations. One can check the corresponding properties of $(N^*)^{-1}N_{11}$ and construct the matrices $N_{20}, N_{21}$ in the same way. 
\end{example}

\section{Finding roots from multiplication tables} \label{sec:tablestosolutions}
Before showing some more experiments, we discuss how to find the $\D$
solutions from the output of the algorithms in this paper using
algorithms from numerical linear algebra. To give a general
description, suppose $m_{g_i}, i = 1,\ldots,n$ are the matrices
corresponding to multiplication by the $n$ generators $g_i$ of a
$\C$-algebra $A$ (be it $R/I, R/I^*$ or $\Rh_\rho/I_\rho\sim
\C[\frac{x_{0}}{h},\ldots, \frac{x_{n}}{h}]/\tilde{I}$ ) in some basis. These matrices share a set of $\D_0$ invariant subspaces, each associated to one of the isolated solutions in $\V(I)$ \cite{elkadi_introduction_2007}. We treat the case of simple roots and the case of roots with multiplicities $\mu_i > 1$ separately.

\subsection{Simple roots: simultaneous diagonalization}
The matrices $m_{g_1}, \ldots, m_{g_n}$ commute and have common eigenvectors. The eigenvalues of $m_{g_i}$ are $g_i(z_j), j = 1, \ldots, \D$. The $m_{g_i}$ can be diagonalized simultaneously. We can compute the common eigenvectors by diagonalizing a generic linear combination $m^*$ of the $m_{g_i}$: $m^* = h(m_{g_1},\ldots, m_{g_n})= \sum_{i=1}^n h_i m_{g_i}$, such that with probability one, all of the eigenvalues $h(g_1, \ldots, g_n)(z_j), j= 1, \ldots, \D$ are distinct and the eigenvectors are well separated. Let $g^* = h(g_1, \ldots, g_n)$, we find $Pm^*P^{-1} = J^*$ with $J^* = \diag(g^*(z_1), \ldots, g^*(z_\D))$. Applying the same transformation to the $m_{g_i}$ gives $Pm_{g_i} P^{-1} = \diag(g_i(z_1), \ldots, g_i(z_\D))$ where the order of the roots corresponding to the diagonal elements is preserved. If the $g_i$ are coordinate functions, we can read off the coordinates of the $\D$ roots from the diagonals of the $Pm_{g_i} P^{-1}$. 

We note that a simultaneous diagonalization of a set of commuting matrices in the non defective case is equivalent to the tensor rank decomposition of a third order tensor \cite{dL2006}. It is possible to use tensor algorithms to refine the solutions obtained by the algorithm described above. The routine \texttt{cpd\_gevd} in Tensorlab can be used for this computation \cite{vervliet2016tensorlab}. 

An alternative is to compute the complex Schur form of $m^*$: $Um^* U^H = T^*$, with $U$ orthogonal, $T^*$ upper triangular and $\cdot ^H$ denotes the Hermitian transpose. The same transformation makes the $m_{g_i}$ upper triangular: $Um_{g_i}U^H = T_i$ and the solutions can be read off from the diagonals of the $T_i$. 
\subsection{Multiple roots: simultaneous block triangularization}
We compute the Jordan form of $m^*$. Let $P m^* P^{-1} = J^*$ with 
$$ J^* = \begin{bmatrix}
J^*_1 \\ & J^*_2 \\ & & \ddots \\ & & & J^*_{\D_0}
\end{bmatrix} = \diag(J^*_1,\ldots,J^*_{\D_0}),$$
such that $J^*_i$ is of size $\mu_i \times \mu_i$, upper triangular with diagonal elements all equal to $g^*(z_i)$. Then $Pm_{g_i}P^{-1} = J_i= \diag (J_{i1}, \ldots, J_{i\D_0})$ with $J_{ij}$ of size $\mu_j \times \mu_j$, upper triangular with diagonal elements equal to $g_i(z_j)$.\footnote{Note that the $J_i$ are not necessarily a Jordan form of the $m_{g_i}$, they may have a different upper triangular nonzero structure than just an upper diagonal of ones \cite{elkadi_introduction_2007,stetter}.} This way, the solutions, along with their multiplicities, can be found from this simultaneous upper triangularization of the $m_{g_i}$. Unfortunately, the Jordan form of a defective matrix is very ill conditioned and its computation is not possible in finite precision arithmetic.

Since we are interested in numerical methods using finite precision arithmetic, we use the following alternative method \cite{corless1997reordered}. We compute the Schur form of $m^*$: $\tilde{U} m^* \tilde{U}^H = \tilde{T}^*$, with $\tilde{U}$ orthogonal and $\tilde{T}^*$ upper triangular. If there are solutions with multiplicity $> 1$, some elements on the diagonal of $\tilde{T}^*$ appear multiple times. Next, we use a clustering of the diagonal elements of $\tilde{T}^*$ and reorder the factorization $U m^* U^H = T^*$ such that $U$ is orthogonal, $T^*$ is upper triangular and the diagonal elements are clustered. The same transformation makes the $m_{g_i}$ block upper triangular with $\D_0$ diagonal blocks of size $\mu_j \times \mu_j, j = 1, \ldots, \D_0$ corresponding to the clusters on the diagonal of $T^*$. All of the diagonal blocks only have one eigenvalue, which is $g_i(z_j)$. For more details on this approach we refer to \cite{corless1997reordered}. Another approach based on the intersection of eigenspaces is given in \cite{moller2001multivariate} and \cite{graillat2009new}.

\section{Numerical examples} \label{sec:numexp}

We give a few more examples in which we use the algorithms in this paper to solve bigger systems. All computations are performed using Matlab on an 8 GB RAM machine with an intel Core i7-6820HQ CPU working at 2.70 GHz. To measure the quality of the solutions, we use the residual as defined in \cite{telen2017stabilized}. 
 
\subsection{Affine solutions of a sparse 3-variate system}
We consider the system given by 
\begin{eqnarray}
f_1 &=& 12x_1x_2x_3^{12} + 7x_1^2x_2^7x_3^6 + 4x_1^{10}x_2^{11}x_3^8 +4x_1^{6}x_2^4x_3^7+5,\\
f_2 &=& 15x_1^{10}x_2^4x_3^2+4x_1^3x_2^6x_3^6+10x_1x_2^{10}x_3^8 + 11x_1^6x_2^{11}x_3^8+12,\\
f_3 &=& 10x_1^7x_2^4x_3^6+4x_1^{10}x_2x_3 + 4x_1^2x_2^{12}x_3^9+14x_1^{10}x_2^5x_3+2.
\end{eqnarray}
The mixed volume (computed using PHCpack \cite{verschelde1999algorithm}) is 2352. Constructing the Macaulay matrix supported in $\sum_{i=1}^3 P_i + \Delta_3 + v$ where $\Delta_3$ is the simplex in $\R^3$ and $v$ is a random small vector, Algorithm \ref{alg:toric} finds 2352 solutions, 2 of which are real. All solutions lie in $(\C^{*})^{3}$, so in this example $I = I^*$. The real solutions are depicted in Figure \ref{fig:exsparse} together with a picture of the surfaces defined by the $f_i$ in $\R^3$.
\begin{figure}
\centering
\includegraphics[scale=0.7]{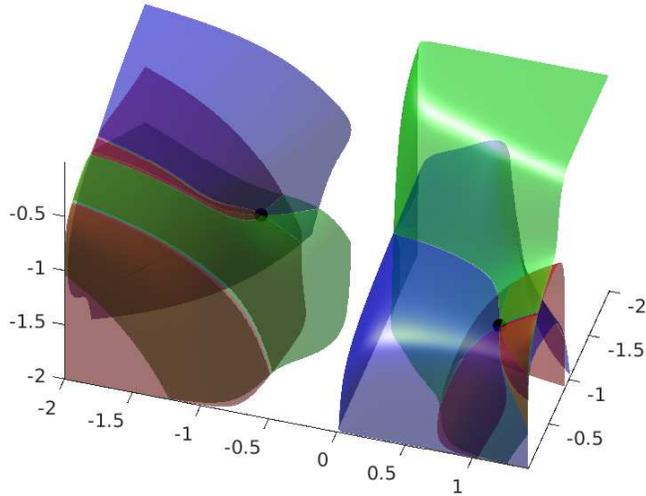}
\caption{Surfaces in $\R^3$ defined by $f_1, f_2$ and $f_3$ (blue, red, green respectively) and the real solutions found using Algorithm \ref{alg:toric}.}
\label{fig:exsparse}
\end{figure}
All solutions are simple. They are found by a Schur decomposition of the $m_{x_i}, i= 1,\ldots, 3$. Computations with polytopes (except for the mixed volume) are done using polymake \cite{polymake:FPSAC_2009}. We used QR with optimal column pivoting on $N_{|W}$ for the basis choice \cite{telen2017stabilized}. The total computation time is about 294 seconds. All solutions are found with a residual smaller than $3.1 \cdot 10^{-12}$.

\subsection{Affine solutions of a generic dense system}
We consider generic dense systems in the sense of \cite{telen2017stabilized}. We compute the solution by decomposing the tensor defined by the $N_i$ from Algorithm \ref{alg:affine} and choose the basis using QR with pivoting. For this type of systems, the basis choice made in Algorithm \ref{alg:affine} agrees with the basis chosen in \cite{telen2017stabilized}. Also here, the monomials at the border of the support are preferred. Figure \ref{fig:basis} shows the basis that is selected for a bivariate system with $d_1=d_2=15$. 

\begin{figure}
\centering
%
%
\begin{tikzpicture}

\begin{axis}[%
width=2.5in,
height=2.5in,
at={(0.48in,0.366in)},
scale only axis,
xmin=0,
xmax=30,
xlabel = $\deg(x_1)$,
ymin=0,
ymax=30,
ylabel = $\deg(x_2)$,
axis background/.style={fill=white}
]
\addplot [color=red, only marks, draw=none, mark=*, mark options={solid, red}, forget plot]
  table[row sep=crcr]{%
0	0\\
0	1\\
0	2\\
0	3\\
0	4\\
0	5\\
1	0\\
28	0\\
0	28\\
1	1\\
27	1\\
2	0\\
27	0\\
0	27\\
0	6\\
26	0\\
0	7\\
0	26\\
0	25\\
1	2\\
3	0\\
1	27\\
2	1\\
4	0\\
1	3\\
26	1\\
25	0\\
0	8\\
24	0\\
0	24\\
23	0\\
2	26\\
1	26\\
22	0\\
26	2\\
3	1\\
0	23\\
25	1\\
0	9\\
0	22\\
21	0\\
0	12\\
3	25\\
0	10\\
2	25\\
1	5\\
1	25\\
5	0\\
2	2\\
20	0\\
0	11\\
25	3\\
24	1\\
1	4\\
4	24\\
1	24\\
4	1\\
0	21\\
8	0\\
6	0\\
3	24\\
7	0\\
24	4\\
2	3\\
9	0\\
2	24\\
1	6\\
10	0\\
0	13\\
23	1\\
2	23\\
5	23\\
23	5\\
2	6\\
25	2\\
19	0\\
0	20\\
17	0\\
18	0\\
22	1\\
11	0\\
1	23\\
22	6\\
2	4\\
2	5\\
0	14\\
3	23\\
0	18\\
5	1\\
2	22\\
21	1\\
1	22\\
1	8\\
4	23\\
24	3\\
3	2\\
0	19\\
6	22\\
23	4\\
3	22\\
12	0\\
2	21\\
0	17\\
1	21\\
16	0\\
1	7\\
0	15\\
14	0\\
2	7\\
5	22\\
22	5\\
20	8\\
24	2\\
7	21\\
6	1\\
1	9\\
21	7\\
0	16\\
2	20\\
20	1\\
22	2\\
21	6\\
2	8\\
4	2\\
4	22\\
3	21\\
13	0\\
1	10\\
23	2\\
1	20\\
9	19\\
23	3\\
3	3\\
2	19\\
19	1\\
8	20\\
1	12\\
8	1\\
21	5\\
20	7\\
4	21\\
21	2\\
3	20\\
5	2\\
2	18\\
12	16\\
6	21\\
20	5\\
17	11\\
9	1\\
19	9\\
10	18\\
20	6\\
7	1\\
3	19\\
16	12\\
17	1\\
14	14\\
1	19\\
18	1\\
15	0\\
16	1\\
3	6\\
11	17\\
4	5\\
19	2\\
18	2\\
3	5\\
13	15\\
2	9\\
5	21\\
1	18\\
20	2\\
19	8\\
19	6\\
1	16\\
10	1\\
2	15\\
15	13\\
19	5\\
5	20\\
18	10\\
4	6\\
12	1\\
6	2\\
1	15\\
19	7\\
14	1\\
15	1\\
8	19\\
22	4\\
4	20\\
7	19\\
17	2\\
17	8\\
5	3\\
22	3\\
3	18\\
11	1\\
2	17\\
2	11\\
15	12\\
17	6\\
14	13\\
21	4\\
18	5\\
9	18\\
16	5\\
7	17\\
4	8\\
1	11\\
11	16\\
1	14\\
18	3\\
1	13\\
6	18\\
14	2\\
4	18\\
8	5\\
18	6\\
4	17\\
13	13\\
9	15\\
9	17\\
3	17\\
}; \label{reddots}
\addplot [color=black, only marks, draw=none, mark=o, mark options={solid, black}, forget plot]
  table[row sep=crcr]{%
0	0\\
1	0\\
0	1\\
2	0\\
1	1\\
0	2\\
3	0\\
2	1\\
1	2\\
0	3\\
4	0\\
3	1\\
2	2\\
1	3\\
0	4\\
5	0\\
4	1\\
3	2\\
2	3\\
1	4\\
0	5\\
6	0\\
5	1\\
4	2\\
3	3\\
2	4\\
1	5\\
0	6\\
7	0\\
6	1\\
5	2\\
4	3\\
3	4\\
2	5\\
1	6\\
0	7\\
8	0\\
7	1\\
6	2\\
5	3\\
4	4\\
3	5\\
2	6\\
1	7\\
0	8\\
9	0\\
8	1\\
7	2\\
6	3\\
5	4\\
4	5\\
3	6\\
2	7\\
1	8\\
0	9\\
10	0\\
9	1\\
8	2\\
7	3\\
6	4\\
5	5\\
4	6\\
3	7\\
2	8\\
1	9\\
0	10\\
11	0\\
10	1\\
9	2\\
8	3\\
7	4\\
6	5\\
5	6\\
4	7\\
3	8\\
2	9\\
1	10\\
0	11\\
12	0\\
11	1\\
10	2\\
9	3\\
8	4\\
7	5\\
6	6\\
5	7\\
4	8\\
3	9\\
2	10\\
1	11\\
0	12\\
13	0\\
12	1\\
11	2\\
10	3\\
9	4\\
8	5\\
7	6\\
6	7\\
5	8\\
4	9\\
3	10\\
2	11\\
1	12\\
0	13\\
14	0\\
13	1\\
12	2\\
11	3\\
10	4\\
9	5\\
8	6\\
7	7\\
6	8\\
5	9\\
4	10\\
3	11\\
2	12\\
1	13\\
0	14\\
15	0\\
14	1\\
13	2\\
12	3\\
11	4\\
10	5\\
9	6\\
8	7\\
7	8\\
6	9\\
5	10\\
4	11\\
3	12\\
2	13\\
1	14\\
0	15\\
16	0\\
15	1\\
14	2\\
13	3\\
12	4\\
11	5\\
10	6\\
9	7\\
8	8\\
7	9\\
6	10\\
5	11\\
4	12\\
3	13\\
2	14\\
1	15\\
0	16\\
17	0\\
16	1\\
15	2\\
14	3\\
13	4\\
12	5\\
11	6\\
10	7\\
9	8\\
8	9\\
7	10\\
6	11\\
5	12\\
4	13\\
3	14\\
2	15\\
1	16\\
0	17\\
18	0\\
17	1\\
16	2\\
15	3\\
14	4\\
13	5\\
12	6\\
11	7\\
10	8\\
9	9\\
8	10\\
7	11\\
6	12\\
5	13\\
4	14\\
3	15\\
2	16\\
1	17\\
0	18\\
19	0\\
18	1\\
17	2\\
16	3\\
15	4\\
14	5\\
13	6\\
12	7\\
11	8\\
10	9\\
9	10\\
8	11\\
7	12\\
6	13\\
5	14\\
4	15\\
3	16\\
2	17\\
1	18\\
0	19\\
20	0\\
19	1\\
18	2\\
17	3\\
16	4\\
15	5\\
14	6\\
13	7\\
12	8\\
11	9\\
10	10\\
9	11\\
8	12\\
7	13\\
6	14\\
5	15\\
4	16\\
3	17\\
2	18\\
1	19\\
0	20\\
21	0\\
20	1\\
19	2\\
18	3\\
17	4\\
16	5\\
15	6\\
14	7\\
13	8\\
12	9\\
11	10\\
10	11\\
9	12\\
8	13\\
7	14\\
6	15\\
5	16\\
4	17\\
3	18\\
2	19\\
1	20\\
0	21\\
22	0\\
21	1\\
20	2\\
19	3\\
18	4\\
17	5\\
16	6\\
15	7\\
14	8\\
13	9\\
12	10\\
11	11\\
10	12\\
9	13\\
8	14\\
7	15\\
6	16\\
5	17\\
4	18\\
3	19\\
2	20\\
1	21\\
0	22\\
23	0\\
22	1\\
21	2\\
20	3\\
19	4\\
18	5\\
17	6\\
16	7\\
15	8\\
14	9\\
13	10\\
12	11\\
11	12\\
10	13\\
9	14\\
8	15\\
7	16\\
6	17\\
5	18\\
4	19\\
3	20\\
2	21\\
1	22\\
0	23\\
24	0\\
23	1\\
22	2\\
21	3\\
20	4\\
19	5\\
18	6\\
17	7\\
16	8\\
15	9\\
14	10\\
13	11\\
12	12\\
11	13\\
10	14\\
9	15\\
8	16\\
7	17\\
6	18\\
5	19\\
4	20\\
3	21\\
2	22\\
1	23\\
0	24\\
25	0\\
24	1\\
23	2\\
22	3\\
21	4\\
20	5\\
19	6\\
18	7\\
17	8\\
16	9\\
15	10\\
14	11\\
13	12\\
12	13\\
11	14\\
10	15\\
9	16\\
8	17\\
7	18\\
6	19\\
5	20\\
4	21\\
3	22\\
2	23\\
1	24\\
0	25\\
26	0\\
25	1\\
24	2\\
23	3\\
22	4\\
21	5\\
20	6\\
19	7\\
18	8\\
17	9\\
16	10\\
15	11\\
14	12\\
13	13\\
12	14\\
11	15\\
10	16\\
9	17\\
8	18\\
7	19\\
6	20\\
5	21\\
4	22\\
3	23\\
2	24\\
1	25\\
0	26\\
27	0\\
26	1\\
25	2\\
24	3\\
23	4\\
22	5\\
21	6\\
20	7\\
19	8\\
18	9\\
17	10\\
16	11\\
15	12\\
14	13\\
13	14\\
12	15\\
11	16\\
10	17\\
9	18\\
8	19\\
7	20\\
6	21\\
5	22\\
4	23\\
3	24\\
2	25\\
1	26\\
0	27\\
28	0\\
27	1\\
26	2\\
25	3\\
24	4\\
23	5\\
22	6\\
21	7\\
20	8\\
19	9\\
18	10\\
17	11\\
16	12\\
15	13\\
14	14\\
13	15\\
12	16\\
11	17\\
10	18\\
9	19\\
8	20\\
7	21\\
6	22\\
5	23\\
4	24\\
3	25\\
2	26\\
1	27\\
0	28\\
29	0\\
28	1\\
27	2\\
26	3\\
25	4\\
24	5\\
23	6\\
22	7\\
21	8\\
20	9\\
19	10\\
18	11\\
17	12\\
16	13\\
15	14\\
14	15\\
13	16\\
12	17\\
11	18\\
10	19\\
9	20\\
8	21\\
7	22\\
6	23\\
5	24\\
4	25\\
3	26\\
2	27\\
1	28\\
0	29\\
}; \label{blackcircles}
\end{axis}
\end{tikzpicture}%
\caption{Support of the Macaulay matrix (\ref{blackcircles}) and basis of $R/I$ (\ref{reddots}) chosen by Algorithm \ref{alg:affine} for a generic dense bivariate system with $d_1=d_2=15$. The bivariate monomials are identified with $\Z^2$ in the usual way.}
\label{fig:basis}
\end{figure}
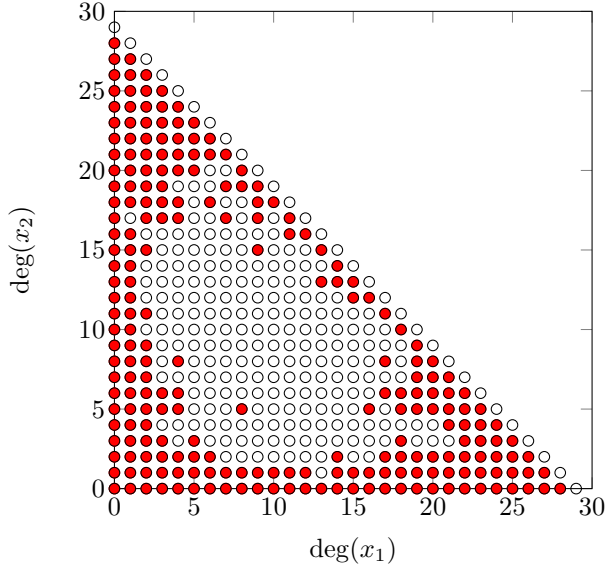

\subsection{Projective solution of a dense system with a solution at infinity} \label{numex:proj}
We use Algorithm \ref{alg:proj} to find the projective coordinates of the 77 solutions in $\PP^2$ of a bivariate system with $d_1 = 7,d_2=11$. There are 7 real solutions, one of which lies at infinity: $(0,1,1)$ where the first coordinate corresponds to the homogenization variable. The algorithm returns all of the projective solutions with a residual smaller than $5.24 \cdot 10^{-14}$ within about a tenth of a second. The left part of Figure \ref{fig:exproj} shows the real solutions in the affine chart $x_0=1$ of $\PP^2$ (there are 6). The right part of the figure shows all real solutions in $\PP^2$ represented as rays connecting the origin in $\C^3$ with a point on the unit sphere. Note that one of the rays (the bold one) is contained in the plane at infinity, and it is also contained in the plane $x_1-x_2=0$, which corresponds to the solution $(0,1,1)$.
\begin{figure}[ht!]
\centering
\input{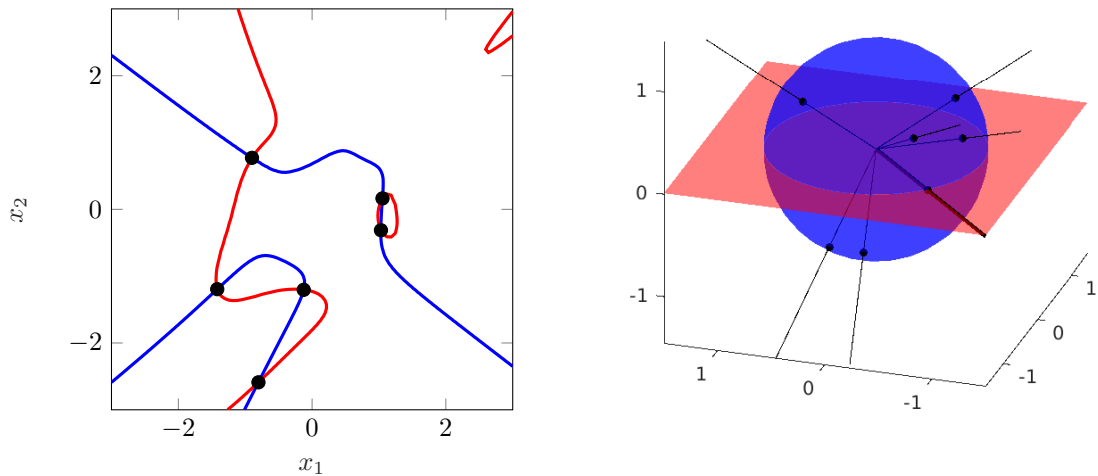}
\caption{Left: picture in $\R^2$ of the solution set in the affine chart $x_0=1$ of $\PP^2$ of the system described in Example \ref{numex:proj}. Right: visualization of the real solutions in $\PP^2$ with the unit sphere in blue and the plane `at infinity' in red.}
\label{fig:exproj}
\end{figure}

\subsection{An example in $\PP^1 \times \PP^1$} \label{ex:multihom}
We consider a system defined by two bivariate affine equations of bidegree $(9,9)$ and $(9,9)$, having 6  solutions `at infinity'. Three of the infinite solutions have an infinite $x_1$-coordinate and a finite $x_2$-coordinate (they are on the divisor $x_{10} = 0$), the others have an infinite $x_2$-coordinate. It takes Algorithm \ref{alg:multihom} about half a second to find all 162 solutions in $\PP^1 \times \PP^1$. The residuals are presented in Figure \ref{fig:multihom}, together with the absolute value of the coordinates of the solutions, dehomogenized with respect to $x_{10}$ and $x_{20}$ respectively.

\begin{figure}[ht!]
\centering
\input{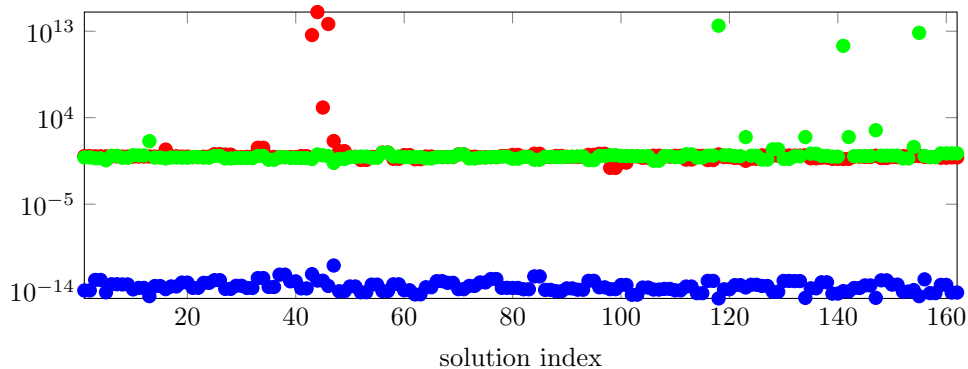}
\caption{Residual (\ref{bluedot}), absolute value of the $x_1$-components (\ref{reddots}) and absolute value of the $x_2$-components (\ref{greendots}) of all 162 numerical solutions of the problem described in Example \ref{ex:multihom}.}
\label{fig:multihom}
\end{figure}

\subsection{Comparison with homotopy solvers}
We compare the speed and accuracy of our method to that of the
homotopy continuation method implemented in PHCpack
\cite{verschelde1999algorithm} and Bertini
\cite{bates2013numerically}.
The current implementation of our method is in Matlab.
We have implemented the construction of the matrix $M$ in Fortran. We
call the routine from Matlab using a MEX file.
An implementation in Julia has also been developed and 
is accessible at \url{https://gitlab.inria.fr/AlgebraicGeometricModeling/AlgebraicSolvers.jl}.

We use double precision for all computations and standard settings for Bertini and PHCpack apart from that. By a generic dense system of degree $d$ in $n$ variables we mean a set of $n$ polynomials in $\C[x_1, \ldots, x_n]$ supported in the monomials $x^\alpha$ of degree $\leq d$ with coefficients drawn from a normal distribution with mean zero and standard deviation 1. For the experiment we fix a value of $n$ and generate generic dense systems of increasing degree $d$ to use as input for the different solvers. 

Tables \ref{t21} up to \ref{t52} give detailed results from the experiment. The following notation is used in the tables. The number of solutions of the input system is $\D$ (in this case, $\D = d^n$). The numbers $m_1, m_2 = n_1, n_2$ give the sizes of $M$ and $N$ from the algorithms: $M^\top \in \C^{m_1 \times m_2}, N \in \C^{n_1 \times n_2}$. The maximal residual of the solutions computed by the algebraic solver of this paper is denoted by res. The number of solutions found by the different solvers is $\D_{\textup{alg}}, \D_{\textup{phc}}, \D_{\textup{brt}}$ for the algebraic solver, PHCpack and Bertini respectively. Since the homotopy methods use Newton refinement
intrinsically, their computed solutions give residuals of the order of
the unit roundoff. The values $t_M, t_N, t_B, t_S$ denote the time for the construction of the Macaulay matrix (Fortran), the computation of its null space, the computation of the basis via QR together with the construction of the multiplication matrices and the time to compute the simultaneous Schur decomposition respectively. The total computation times are $t_{\textup{alg}}, t_{\textup{phc}}$ and $t_{\textup{brt}}$ for the algebraic solver introduced in this paper ($t_{\textup{alg}} = t_M + t_N + t_B + t_S$), PHCpack and Bertini respectively. All timings are in seconds. Tables \ref{t21} and \ref{t22} present the experiment for $n = 2$ variables, Tables \ref{t31} and \ref{t32} for $n = 3$, Tables \ref{t41} and \ref{t42} for $n = 4$ and Tables \ref{t51} and \ref{t52} for $n = 5$.

We observe that our method has found numerical approximations for
\textit{all} $d^n$ roots, with a residual no larger than order $10^{-9}$. Due
to the quadratic convergence of Newton's iteration, one refining step
can be expected to result in a residual of the order of the unit
roundoff. Table \ref{t21} shows that for 2
variables, up to degree $d = 61$, our method is the fastest. For $n=3$
this is no longer the case but timings are comparable. For a larger
number of variables, the matrix $M$ in the algorithms becomes very
large and the null space computation is expensive, which makes the
algebraic method slower than the continuation solvers.

\begin{table}[h!]
	\centering
	\footnotesize
	\begin{filecontents*}{T21.txt}
Var1 nbsol m1 m2 n1 n2 res nbsol_qr nbsol_phc nbsol_brt
1 1 2 3 3 1 1.28350582708179e-16 1 1 1
7 49 56 105 105 49 2.06207660687103e-13 49 49 49
13 169 182 351 351 169 2.18318606944498e-13 169 169 169
19 361 380 741 741 361 5.27930295875126e-13 361 361 361
25 625 650 1275 1275 625 1.20816080110579e-10 625 614 625
31 961 992 1953 1953 961 5.22868007967963e-09 961 951 961
37 1369 1406 2775 2775 1369 4.04623825971634e-12 1369 1360 1368
43 1849 1892 3741 3741 1849 1.74155957399806e-11 1849 1825 1845
49 2401 2450 4851 4851 2401 1.57015062140972e-10 2401 2364 2163
55 3025 3080 6105 6105 3025 1.83760378202494e-11 3025 2970 2487
61 3721 3782 7503 7503 3721 3.2617754715295e-11 3721 3662 2260
\end{filecontents*}
\pgfplotstabletypeset[ 
	every head row/.style={before row=\toprule,after row=\midrule},
	every last row/.style={after row=\bottomrule},
	columns = {Var1,nbsol,m1,m2,n2,res,nbsol_qr,nbsol_phc,nbsol_brt},
	columns/Var1/.style={int detect,column type=c|,column name=$d$},
	columns/nbsol/.style={int detect,column type=c,column name= $\D$},
	columns/m1/.style={column type=c,column name= $m_1$},
	columns/m2/.style={column type=c,column name= $m_2 {=} n_1$},
	columns/n2/.style={column type=c,column name= $n_2$},
	columns/res/.style={column type=c,column name= res},
	columns/nbsol_qr/.style={column type=c,column name= $\D_{\textup{alg}}$},
	columns/nbsol_phc/.style={column type=c,column name= $\D_{\textup{phc}}$},
	columns/nbsol_brt/.style={column type=c,column name= $\D_{\textup{brt}}$},
	]
	{T21.txt}
	\caption{Numerical results for PHCpack, Bertini and our method for dense systems in $n=2$ variables of increasing degree $d$. The table shows matrix sizes, accuracy and number of solutions.}
	\label{t21}
\end{table}

\begin{table}[h!]
	\centering
	\footnotesize
	\begin{filecontents*}{T22.txt}
Var1 nbsol time_mac time_ker time_basis time_EV time_qr time_phc time_brt
1 1 0.000148 5.5e-05 0.000296 3.6e-05 0.000535 0.056047 0.014092
7 49 0.007884 0.001676 0.003758 0.002781 0.016099 0.177394 0.086488
13 169 0.046534 0.010293 0.016615 0.028067 0.101509 0.836549 1.143355
19 361 0.129059 0.056912 0.053438 0.131957 0.371366 3.289626 8.794189
25 625 0.316584 0.182024 0.154553 0.50502 1.158181 8.794417 33.826457
31 961 0.545289 0.506717 0.55363 1.490616 3.096252 20.25331 98.390135
37 1369 0.958571 1.524707 1.500592 3.516539 7.500409 39.920734 258.087237
43 1849 1.470184 4.04501 3.802483 8.278899 17.596576 69.096816 504.00803
49 2401 2.465175 10.461769 8.783241 17.910314 39.620499 124.467662 891.368273
55 3025 3.690334 20.50627 17.846825 34.29518 76.338609 178.547092 1581.767293
61 3721 4.847996 36.316762 31.26111 62.87479 135.300658 283.873427 2115.660379
\end{filecontents*}
\pgfplotstabletypeset[ 
	every head row/.style={before row=\toprule,after row=\midrule},
	every last row/.style={after row=\bottomrule},
	columns = {Var1,time_mac,time_ker,time_basis,time_EV,time_qr,time_phc,time_brt},
	columns/Var1/.style={int detect,column type=c|,column name=$d$},
	columns/time_mac/.style={column type=c,column name= $t_M$},
	columns/time_ker/.style={column type=c,column name= $t_N$},
	columns/time_basis/.style={column type=c,column name= $t_B$},
	columns/time_EV/.style={column type=c,column name= $t_S$},
	columns/time_qr/.style={column type=c,column name= $t_{\textup{alg}}$},
	columns/time_phc/.style={column type=c,column name= $t_{\textup{phc}}$},
	columns/time_brt/.style={column type=c,column name= $t_{\textup{brt}}$},
	]
	{T22.txt}
	\caption{Timing results for PHCpack, Bertini and our method for dense systems in $n=2$ variables of increasing degree $d$.}
	\label{t22}
\end{table}

\begin{table}[h!]
	\centering
	\footnotesize
	\begin{filecontents*}{T31.txt}
Var1 nbsol m1 m2 n1 n2 res nbsol_qr nbsol_phc nbsol_brt
1 1 3 4 4 1 1.78542566199419e-16 1 1 1
3 27 105 120 120 27 1.05293263507094e-14 27 27 27
5 125 495 560 560 125 1.2907023109606e-12 125 125 125
7 343 1365 1540 1540 343 6.71040264252613e-12 343 343 343
9 729 2907 3276 3276 729 1.38493324777182e-10 729 726 729
11 1331 5313 5984 5984 1331 3.1137165101523e-11 1331 1331 1331
13 2197 8775 9880 9880 2197 2.86160870190343e-11 2197 2192 2197
\end{filecontents*}
\pgfplotstabletypeset[ 
	every head row/.style={before row=\toprule,after row=\midrule},
	every last row/.style={after row=\bottomrule},
	columns = {Var1,nbsol,m1,m2,n2,res,nbsol_qr,nbsol_phc,nbsol_brt},
	columns/Var1/.style={int detect,column type=c|,column name=$d$},
	columns/nbsol/.style={int detect,column type=c,column name= $\D$},
	columns/m1/.style={column type=c,column name= $m_1$},
	columns/m2/.style={column type=c,column name= $m_2{=}n_1$},
	columns/n2/.style={column type=c,column name= $n_2$},
	columns/res/.style={column type=c,column name= res},
	columns/nbsol_qr/.style={column type=c,column name= $\D_{\textup{alg}}$},
	columns/nbsol_phc/.style={column type=c,column name= $\D_{\textup{phc}}$},
	columns/nbsol_brt/.style={column type=c,column name= $\D_{\textup{brt}}$},
	]
	{T31.txt}
	\caption{Numerical results for PHCpack, Bertini and our method for dense systems in $n=3$ variables of increasing degree $d$. The table shows matrix sizes, accuracy and number of solutions.}
	\label{t31}
\end{table}

\begin{table}[h!]
	\centering
	\footnotesize
	\begin{filecontents*}{T32.txt}
Var1 nbsol time_mac time_ker time_basis time_EV time_qr time_phc time_brt
1 1 0.000372 0.000124 0.00231 4.5e-05 0.002851 0.068003 0.016938
3 27 0.007905 0.002419 0.007063 0.001077 0.018464 0.144008 0.073307
5 125 0.05664 0.039328 0.033104 0.011746 0.140818 0.681384 0.625988
7 343 0.231036 1.126188 0.117283 0.099017 1.573524 3.418974 4.105194
9 729 0.682535 14.43443 0.650983 0.628444 16.396392 12.212029 17.293282
11 1331 1.771619 44.792444 3.910334 3.982299 54.456696 39.084044 70.662765
13 2197 5.811976 183.667085 16.071985 15.351092 220.902138 97.277697 210.337602
\end{filecontents*}
\pgfplotstabletypeset[ 
	every head row/.style={before row=\toprule,after row=\midrule},
	every last row/.style={after row=\bottomrule},
	columns = {Var1,time_mac,time_ker,time_basis,time_EV,time_qr,time_phc,time_brt},
	columns/Var1/.style={int detect,column type=c|,column name=$d$},
	columns/time_mac/.style={column type=c,column name= $t_M$},
	columns/time_ker/.style={column type=c,column name= $t_N$},
	columns/time_basis/.style={column type=c,column name= $t_B$},
	columns/time_EV/.style={column type=c,column name= $t_S$},
	columns/time_qr/.style={column type=c,column name= $t_{\textup{alg}}$},
	columns/time_phc/.style={column type=c,column name= $t_{\textup{phc}}$},
	columns/time_brt/.style={column type=c,column name= $t_{\textup{brt}}$},
	]
	{T32.txt}
	\caption{Timing results for PHCpack, Bertini and our method for dense systems in $n=3$ variables of increasing degree $d$.}
	\label{t32}
\end{table}

\begin{table}[h!]
	\centering
	\footnotesize
	\begin{filecontents*}{T41.txt}
Var1 nbsol m1 m2 n1 n2 res nbsol_qr nbsol_phc nbsol_brt
1 1 4 5 5 1 1.24415169788888e-16 1 1 1
2 16 140 126 126 16 1.12706293691187e-14 16 16 16
3 81 840 715 715 81 3.84482536934835e-14 81 81 81
4 256 2860 2380 2380 256 1.51561867054249e-13 256 256 255
\end{filecontents*}
\pgfplotstabletypeset[ 
	every head row/.style={before row=\toprule,after row=\midrule},
	every last row/.style={after row=\bottomrule},
	columns = {Var1,nbsol,m1,m2,n2,res,nbsol_qr,nbsol_phc,nbsol_brt},
	columns/Var1/.style={int detect,column type=c|,column name=$d$},
	columns/nbsol/.style={int detect,column type=c,column name= $\D$},
	columns/m1/.style={column type=c,column name= $m_1$},
	columns/m2/.style={column type=c,column name= $m_2{=}n_1$},
	columns/n2/.style={column type=c,column name= $n_2$},
	columns/res/.style={column type=c,column name= res},
	columns/nbsol_qr/.style={column type=c,column name= $\D_{\textup{alg}}$},
	columns/nbsol_phc/.style={column type=c,column name= $\D_{\textup{phc}}$},
	columns/nbsol_brt/.style={column type=c,column name= $\D_{\textup{brt}}$},
	]
	{T41.txt}
	\caption{Numerical results for PHCpack, Bertini and our method for dense systems in $n=4$ variables of increasing degree $d$. The table shows matrix sizes, accuracy and number of solutions.}
	\label{t41}
\end{table}

\begin{table}[h!]
	\centering
	\footnotesize
	\begin{filecontents*}{T42.txt}
Var1 nbsol time_mac time_ker time_basis time_EV time_qr time_phc time_brt
1 1 0.011035 0.000283 0.018288 0.000843 0.030449 0.068229 0.017648
2 16 0.011172 0.004291 0.010798 0.000594 0.026855 0.124923 0.063197
3 81 0.112232 0.136392 0.057617 0.005547 0.311788 0.52403 0.586863
4 256 0.456604 8.312964 0.229042 0.054139 9.052749 2.271787 3.624115
\end{filecontents*}
\pgfplotstabletypeset[ 
	every head row/.style={before row=\toprule,after row=\midrule},
	every last row/.style={after row=\bottomrule},
	columns = {Var1,time_mac,time_ker,time_basis,time_EV,time_qr,time_phc,time_brt},
	columns/Var1/.style={int detect,column type=c|,column name=$d$},
	columns/time_mac/.style={column type=c,column name= $t_M$},
	columns/time_ker/.style={column type=c,column name= $t_N$},
	columns/time_basis/.style={column type=c,column name= $t_B$},
	columns/time_EV/.style={column type=c,column name= $t_S$},
	columns/time_qr/.style={column type=c,column name= $t_{\textup{alg}}$},
	columns/time_phc/.style={column type=c,column name= $t_{\textup{phc}}$},
	columns/time_brt/.style={column type=c,column name= $t_{\textup{brt}}$},
	]
	{T42.txt}
	\caption{Timing results for PHCpack, Bertini and our method for dense systems in $n=4$ variables of increasing degree $d$.}
	\label{t42}
\end{table}

\begin{table}[h!]
	\centering
	\footnotesize
	\begin{filecontents*}{T51.txt}
Var1 nbsol m1 m2 n1 n2 res nbsol_qr nbsol_phc nbsol_brt
1 1 5 6 6 1 7.89469484589712e-17 1 1 1
2 32 630 462 462 32 4.21990984425128e-14 32 32 32
3 243 6435 4368 4368 243 1.8436883977569e-12 243 243 243
\end{filecontents*}
\pgfplotstabletypeset[ 
	every head row/.style={before row=\toprule,after row=\midrule},
	every last row/.style={after row=\bottomrule},
	columns = {Var1,nbsol,m1,m2,n2,res,nbsol_qr,nbsol_phc,nbsol_brt},
	columns/Var1/.style={int detect,column type=c|,column name=$d$},
	columns/nbsol/.style={int detect,column type=c,column name= $\D$},
	columns/m1/.style={column type=c,column name= $m_1$},
	columns/m2/.style={column type=c,column name= $m_2{=}n_1$},
	columns/n2/.style={column type=c,column name= $n_2$},
	columns/res/.style={column type=c,column name= res},
	columns/nbsol_qr/.style={column type=c,column name= $\D_{\textup{alg}}$},
	columns/nbsol_phc/.style={column type=c,column name= $\D_{\textup{phc}}$},
	columns/nbsol_brt/.style={column type=c,column name= $\D_{\textup{brt}}$},
	]
	{T51.txt}
	\caption{Numerical results for PHCpack, Bertini and our method for dense systems in $n=5$ variables of increasing degree $d$. The table shows matrix sizes, accuracy and number of solutions.}
	\label{t51}
\end{table}

\begin{table}[h!]
	\centering
	\footnotesize
	\begin{filecontents*}{T52.txt}
Var1 nbsol time_mac time_ker time_basis time_EV time_qr time_phc time_brt
1 1 0.000487 0.000154 0.001857 3e-05 0.002528 0.065163 0.019136
2 32 0.059702 0.039047 0.040693 0.001456 0.140898 0.263123 0.24232
3 243 1.20778 69.379074 0.53467 0.055033 71.176557 2.418417 4.742639
\end{filecontents*}
\pgfplotstabletypeset[ 
	every head row/.style={before row=\toprule,after row=\midrule},
	every last row/.style={after row=\bottomrule},
	columns = {Var1,time_mac,time_ker,time_basis,time_EV,time_qr,time_phc,time_brt},
	columns/Var1/.style={int detect,column type=c|,column name=$d$},
	columns/time_mac/.style={column type=c,column name= $t_M$},
	columns/time_ker/.style={column type=c,column name= $t_N$},
	columns/time_basis/.style={column type=c,column name= $t_B$},
	columns/time_EV/.style={column type=c,column name= $t_S$},
	columns/time_qr/.style={column type=c,column name= $t_{\textup{alg}}$},
	columns/time_phc/.style={column type=c,column name= $t_{\textup{phc}}$},
	columns/time_brt/.style={column type=c,column name= $t_{\textup{brt}}$},
	]
	{T52.txt}
	\caption{Timing results for PHCpack, Bertini and our method for dense systems in $n=5$ variables of increasing degree $d$.}
	\label{t52}
\end{table}
An important note is that homotopy methods do not guarantee that all solutions are found. In fact, they lose some solutions for large systems. For $n = 2, d = 55$, Bertini gives up on 538 out of 3025 paths, so about 18\% of the solutions is not found (using default settings). For the same problem, PHCpack loses 2\% of the solutions. 


\section{Conclusion and future work}
We have proposed an algebraic framework for finding a representation of an Artinian quotient ring $R/I$ and we have shown how this leads to a numerical linear algebra method for solving square systems of polynomial equations with solutions in $\C^n, (\C^*)^n, \PP^n$ or $\PP^{n_1} \times \cdots \times \PP^{n_k}$. The choice of basis $\OO$ for $B$ is crucial for the numerical stability of the method. The experiments in Section \ref{sec:numexp} show that we obtain accurate results. The method guarantees, unlike homotopy solvers, that in exact arithmetic all solutions are found under some genericity assumptions. It is competitive in speed with Bertini and PHCpack for a small number of variables. Here are some ideas for future work. 

\begin{itemize}
\item The submatrix $\tilde{M}$ is generically of full rank, as proved by Macaulay. However, we observe that it has some small singular values for generic, large systems, $n \geq 3$. Therefore it has an ill conditioned null space and it is better to use the larger matrix $M$. If we can find a subset of the columns of $M$ that leads to a full rank matrix with `good' singular values, this could speed up the computations.
\item Sparse systems lead to a sparse matrix $M$. It might be useful to exploit this sparsity in the null space computation. 
\item An implementation in C++, Fortran, \ldots would speed up the
  algorithm, exploiting High Performance Computation optimisation. 
\item The method might be extended to ideals defining the union of a finite set of points and a positive dimensional component. 
\item When the dimension is bigger than $n=4$, most of the time is spent
  in computing the null space $N$. A cheaper construction of the map
  $N$ can be investigated.
\end{itemize}

\small

\begin{thebibliography}{10}

\bibitem{bates2013numerically}
D.~J. Bates, J.~D. Hauenstein, A.~J. Sommese, and C.~W. Wampler.
\newblock {\em Numerically solving polynomial systems with Bertini}, volume~25.
\newblock SIAM, 2013.

\bibitem{bayer1987criterion}
D.~Bayer and M.~Stillman.
\newblock A criterion for detecting $m$-regularity.
\newblock {\em Inventiones {M}athematicae}, 87(1):1--11, 1987.

\bibitem{bernstein}
D.~Bernstein.
\newblock The number of roots of a system of equations.
\newblock {\em Functional Anal. Appl.}, 9:1--4, 1975.

\bibitem{buse_resultant_2001}
L.~Bus{\'e}, M.~Elkadi, and B.~Mourrain.
\newblock Resultant over the residual of a complete intersection.
\newblock {\em Journal of Pure and Applied Algebra}, 164(1-2):35--57, Oct.
  2001.

\bibitem{cattani2005solving}
E.~Cattani, D.~A. Cox, G.~Ch{\`e}ze, A.~Dickenstein, M.~Elkadi, I.~Z. Emiris,
  A.~Galligo, A.~Kehrein, M.~Kreuzer, and B.~Mourrain.
\newblock Solving polynomial equations: foundations, algorithms, and
  applications ({A}lgorithms and {C}omputation in {M}athematics).
\newblock 2005.

\bibitem{corless1997reordered}
R.~M. Corless, P.~M. Gianni, and B.~M. Trager.
\newblock A reordered {S}chur factorization method for zero-dimensional
  polynomial systems with multiple roots.
\newblock In {\em Proceedings of the 1997 International Symposium on Symbolic
  and Algebraic Computation}, pages 133--140. ACM, 1997.

\bibitem{cox1}
D.~Cox, J.~Little, and D.~O'shea.
\newblock {\em Ideals, varieties, and algorithms}, volume~3.
\newblock Springer, 1992.

\bibitem{cox2}
D.~A. Cox, J.~Little, and D.~O'shea.
\newblock {\em Using algebraic geometry}, volume 185.
\newblock Springer Science \& Business Media, 2006.

\bibitem{dandrea_poisson_2015}
C.~D'Andrea and M.~Sombra.
\newblock A {{Poisson}} formula for the sparse resultant.
\newblock {\em Proceedings of the London Mathematical Society},
  110(4):932--964, Apr. 2015.

\bibitem{de_boor_ideal_2004}
C.~De~Boor.
\newblock Ideal interpolation.
\newblock {\em Approximation Theory XI: Gatlinburg}, pages 59--91, 2004.

\bibitem{de_boor_ideal_2006}
C.~{de Boor}.
\newblock Ideal interpolation: {{Mourrain}}'s condition vs. {{D}}-invariance.
\newblock {\em Banach Center Publications, Institute of Mathematics, Polish
  Academy of Sciences, Warszawa}, 72:41--47, 2006.

\bibitem{dL2006}
L.~De~Lathauwer.
\newblock A link between the canonical decomposition in multilinear algebra and
  simultaneous matrix diagonalization.
\newblock 28(3):642--666, 2006.

\bibitem{dreesen2012back}
P.~Dreesen, K.~Batselier, and B.~De~Moor.
\newblock Back to the roots: Polynomial system solving, linear algebra, systems
  theory.
\newblock {\em IFAC Proceedings Volumes}, 45(16):1203--1208, 2012.

\bibitem{eder_signature-based_2011}
C.~Eder and J.~E. Perry.
\newblock Signature-based algorithms to compute {{Gr{\"o}bner}} bases.
\newblock In {\em Proceedings of the 36th International Symposium on
  {{Symbolic}} and Algebraic Computation}, pages 99--106. {ACM}, 2011.

\bibitem{eisenbud_geometry_2005}
D.~Eisenbud.
\newblock {\em The Geometry of Syzygies: A Second Course in Commutative Algebra
  and Algebraic Geometry}.
\newblock {Springer}, New York, NY, 2005.
\newblock OCLC: 249751633.

\bibitem{elkadi_introduction_2007}
M.~Elkadi and B.~Mourrain.
\newblock {\em {Introduction {\`a} la r{\'e}solution des syst{\`e}mes
  polynomiaux}}, volume~59 of {\em Math{\'e}matiques et Applications}.
\newblock {Springer}, 2007.

\bibitem{emir1}
I.~Z. Emiris and J.~Canny.
\newblock A practical method for the sparse resultant.
\newblock {\em Proc. ACM Intern. Symp. on Symbolic and Algebraic Computation},
  pages 183--192, 1993.

\bibitem{emiris_matrices_1999}
I.~Z. Emiris and B.~Mourrain.
\newblock Matrices in {{Elimination Theory}}.
\newblock {\em Journal of Symbolic Computation}, 28(1-2):3--44, 1999.

\bibitem{faugere1999new}
J.-C. Faug{\`e}re.
\newblock A new efficient algorithm for computing {G}r{\"o}bner bases ({F}4).
\newblock {\em Journal of pure and applied algebra}, 139(1):61--88, 1999.

\bibitem{fulton1993introduction}
W.~Fulton.
\newblock {\em Introduction to toric varieties}.
\newblock Number 131. Princeton University Press, 1993.

\bibitem{graillat2009new}
S.~Graillat and P.~Tr{\'e}buchet.
\newblock A new algorithm for computing certified numerical approximations of
  the roots of a zero-dimensional system.
\newblock In {\em Proceedings of the 2009 International Symposium on Symbolic
  and Algebraic Computation}, pages 167--174. ACM, 2009.

\bibitem{hustu}
B.~Huber and B.~Sturmfels.
\newblock A polyhedral method for solving sparse polynomial systems.
\newblock {\em Math. Comp}, 64:1541--1555, 1995.

\bibitem{polymake:FPSAC_2009}
M.~Joswig, B.~M\"uller, and A.~Paffenholz.
\newblock {\tt polymake} and lattice polytopes.
\newblock In {\em 21st {I}nternational {C}onference on {F}ormal {P}ower
  {S}eries and {A}lgebraic {C}ombinatorics ({FPSAC} 2009)}, Discrete Math.
  Theor. Comput. Sci. Proc., AK, pages 491--502. Assoc. Discrete Math. Theor.
  Comput. Sci., Nancy, 2009.

\bibitem{kho}
A.~Khovanskii.
\newblock Newton polytopes and toric varieties.
\newblock {\em Functional Anal. Appl.}, 11:289--298, 1977.

\bibitem{kush}
A.~Kushnirenko.
\newblock Newton polytopes and the {B}\'ezout theorem.
\newblock {\em Functional Anal. Appl.}, 10:233--235, 1976.

\bibitem{macaulay_formulae_1902}
F.~S. Macaulay.
\newblock Some formulae in elimination.
\newblock {\em Proceedings of the London Mathematical Society}, 1(1):3--27,
  1902.

\bibitem{macaulay1994algebraic}
F.~S. Macaulay.
\newblock {\em The algebraic theory of modular systems}.
\newblock Cambridge University Press, 1994.

\bibitem{moller_h-bases_2000}
H.~M. M{\"o}ller and T.~Sauer.
\newblock H-bases for polynomial interpolation and system solving.
\newblock {\em Advances in Computational Mathematics}, 12(4):335--362, 2000.

\bibitem{moller2001multivariate}
H.~M. M{\"o}ller and R.~Tenberg.
\newblock Multivariate polynomial system solving using intersections of
  eigenspaces.
\newblock {\em Journal of symbolic computation}, 32(5):513--531, 2001.

\bibitem{mourrain1999new}
B.~Mourrain.
\newblock A {{New Criterion}} for {{Normal Form Algorithms}}.
\newblock In {\em Proceedings of the 13th {{International Symposium}} on
  {{Applied Algebra}}, {{Algebraic Algorithms}} and {{Error}}-{{Correcting
  Codes}}}, LNCS, pages 430--443, London, UK, 1999. {Springer-Verlag}.

\bibitem{mourrain2009subdivision}
B.~Mourrain and J.~P. Pavone.
\newblock Subdivision methods for solving polynomial equations.
\newblock {\em Journal of Symbolic Computation}, 44(3):292--306, 2009.

\bibitem{mourrain_solving_2000}
B.~Mourrain and P.~Trebuchet.
\newblock Solving projective complete intersection faster.
\newblock In {\em Proceedings of the 2000 International Symposium on
  {{Symbolic}} and Algebraic Computation}, pages 234--241. {ACM Press}, 2000.

\bibitem{mourrain_generalized_2005}
B.~Mourrain and P.~Trebuchet.
\newblock Generalized normal forms and polynomial system solving.
\newblock In {\em Proceedings of the 2005 International Symposium on
  {{Symbolic}} and Algebraic Computation}, pages 253--260. {ACM}, 2005.

\bibitem{mourrain_stable_2008}
B.~Mourrain and P.~Tr{\'e}buchet.
\newblock Stable normal forms for polynomial system solving.
\newblock {\em Theoretical Computer Science}, 409(2):229--240, 2008.

\bibitem{mvb}
L.~Sorber, M.~Van~Barel, and L.~De~Lathauwer.
\newblock Numerical solution of bivariate and polyanalytic polynomial systems.
\newblock {\em SIAM J. Num. Anal. 52}, pages 1551--1572, 2014.

\bibitem{stetter}
H.~J. Stetter.
\newblock {\em Numerical Polynomial Algebra}.
\newblock Society for Industrial and Applied Mathematics, 2004.

\bibitem{sturm}
B.~Sturmfels.
\newblock Polynomial equations and convex polytopes.
\newblock {\em The American Mathematical Monthly}, 105:907--922, 1998.

\bibitem{sturmfels2}
B.~Sturmfels.
\newblock {\em Solving Systems of Polynomial Equations}.
\newblock Number~97 in CBMS Regional Conferences. Amer. Math. Soc., 2002.

\bibitem{telen2017stabilized}
S.~Telen and M.~Van~Barel.
\newblock A stabilized normal form algorithm for generic systems of polynomial
  equations.
\newblock {\em Preprint, arXiv:1708.07670}, 2017.

\bibitem{verschelde1999algorithm}
J.~Verschelde.
\newblock Algorithm 795: {PHCpack}: A general-purpose solver for polynomial
  systems by homotopy continuation.
\newblock {\em ACM Transactions on Mathematical Software (TOMS)},
  25(2):251--276, 1999.

\bibitem{vervliet2016tensorlab}
N.~Vervliet, O.~Debals, L.~Sorber, M.~Van~Barel, and L.~De~Lathauwer.
\newblock Tensorlab 3.0.
\newblock {\em available online, URL: \url{www.tensorlab.net}}, 2016.

\end{thebibliography}

\end{document}